\documentclass[11pt]{article}

\usepackage[english]{babel}

\usepackage[utf8]{inputenc}
\usepackage{amsfonts,amssymb,amsmath,amsthm,latexsym,epsfig,amscd,bbm,stmaryrd,mathrsfs,hyperref,fancyhdr}
\usepackage{graphicx,color,tikz,courier,bbm,enumerate,psfrag,wasysym,type1cm,dsfont,geometry,marginnote,afterpage}
\usepackage{tikz}

\numberwithin{equation}{section}

\theoremstyle{plain} 

\newtheorem{theorem}{Theorem}[section]
\newtheorem{corollary}[theorem]{Corollary}
\newtheorem{lemma}[theorem]{Lemma}

\newtheorem{definition}[theorem]{Definition}

\newtheorem{remark}[theorem]{Remark}

\newcommand{\Z}{\mathbb Z}

\newcommand{\N}{\mathbb N}
\newcommand{\R}{\mathbb R}

\newcommand{\dd}{\mathrm{d}}

\newcommand{\prob}{\overset{\mathrm{pr.}}{\longrightarrow}}
\newcommand{\as}{\overset{\mathrm{a.s.}}{\longrightarrow}}
\newcommand{\weak}{\Longrightarrow}

\newcommand{\E}{\mathbb E}
\renewcommand{\P}{\mathbb P}

\newcommand{\CM}{\mathrm{CM}}

\newcommand{\dn}{d^{(G_{n})}}

\newcommand{\Dst}{\Delta^{(\mathrm{st})}}

\begin{document}

\title{The multi-level friendship paradox for sparse random graphs}

\author{
\renewcommand{\thefootnote}{\arabic{footnote}}
Rajat Subhra Hazra, Frank den Hollander, Azadeh Parvaneh
\footnotemark[1]
}

\footnotetext[1]{
Mathematical Institute, Leiden University, Einsteinweg 55, 2333 CC Leiden, The Netherlands.\\ 
{\tt \{r.s.hazra,denholla,s.a.parvaneh.ziabari\}@math.leidenuniv.nl}
}

\maketitle

\begin{abstract}
In \cite{HHP} we analysed the friendship paradox for sparse random graphs. For four classes of random graphs we characterised the empirical distribution of the friendship biases between vertices and their neighbours at distance $1$, proving convergence as $n\to\infty$ to a limiting distribution, with $n$ the number of vertices, and identifying moments and tail exponents of the limiting distribution. In the present paper we look at the multi-level friendship bias between vertices and their neighbours at distance $k \in \N$ obtained via a $k$-step exploration according to a backtracking or a non-backtracking random walk. We identify the limit of empirical distribution of the multi-level friendship biases as $n\to\infty$ and/or $k\to\infty$. We show that for non-backtracking exploration the two limits commute for a large class of sparse random graphs, including those that locally converge to a rooted Galton-Watson tree. In particular, we show that the same limit arises when $k$ depends on $n$, i.e., $k=k_n$, provided $\lim_{n\to\infty} k_n = \infty$ under some mild conditions. We exhibit cases where the two limits do not commute and show the relevance of the mixing time of the exploration.



\medskip\noindent
\emph{Acknowledgement:}
The research in this paper was supported through NWO Gravitation Grant NETWORKS 024.002.003. AP received funding from the European Union's Horizon 2020 research and innovation programme under the Marie Sk\l odowska-Curie Grant Agreement No 101034253. The authors thank Remco van der Hofstad for fruitful discussions.

%%%%%%%%%%%%%%%%%%%%%
\vspace{0.1cm}
\hfill\includegraphics[scale=0.1]{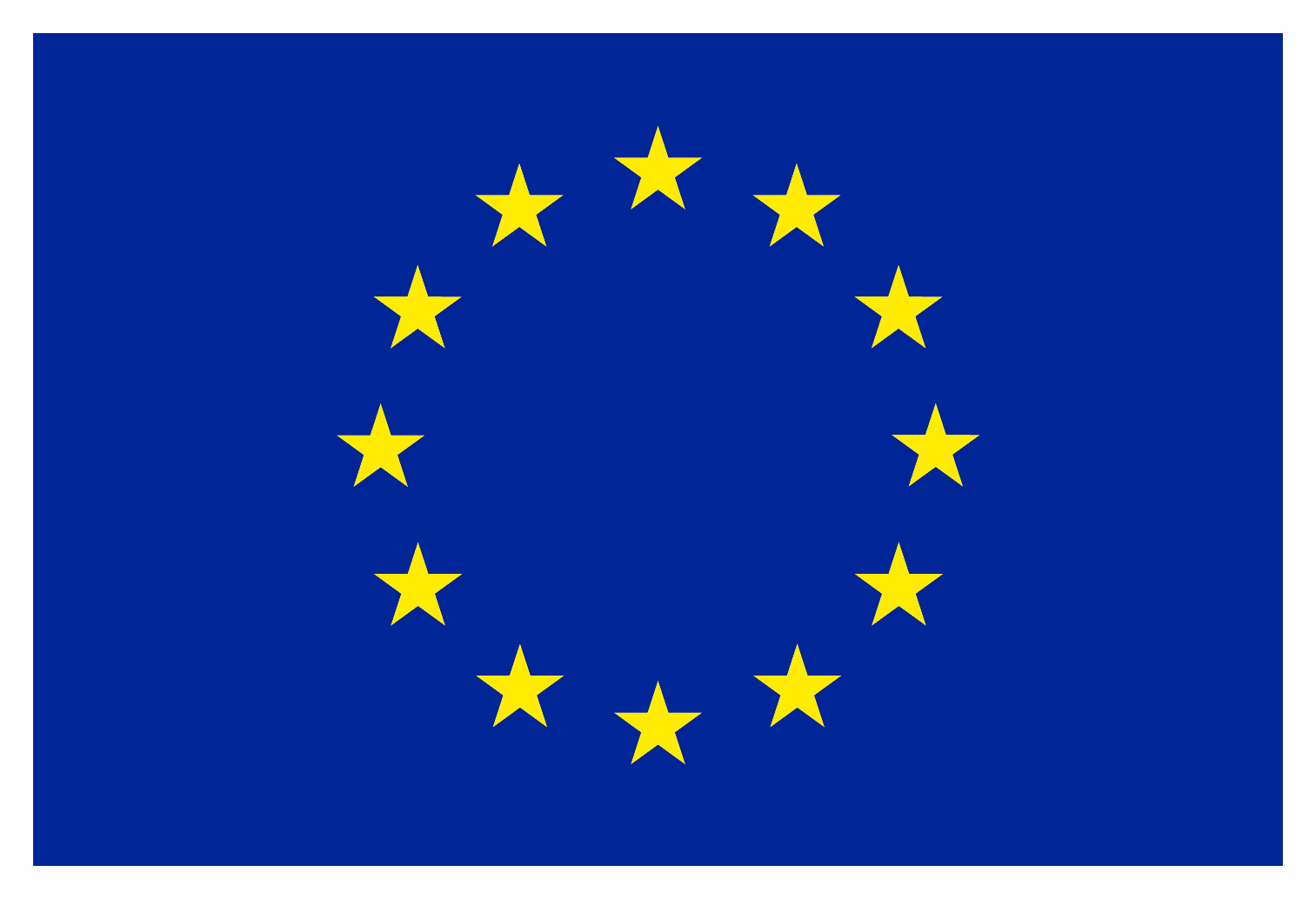}
%%%%%%%%%%%%%%%%%%%%%

\end{abstract}

%%%

\newpage
\scriptsize
\tableofcontents
\normalsize

%%%%% SECTION 1 %%%%%%%%%%%%%%%%%%%

\section{Introduction and outline}
\label{sec:intro}

%%%

\subsection{The friendship paradox}

In 1991, the American sociologist Scott Feld discovered a phenomenon called the \emph{friendship paradox}, which he summarised by saying that `your friends have more friends than you do': within any social network of individuals with mutual friendships the difference between the average number of friends of friends and the average number of friends is always non-negative \cite{SF}. Later, many papers explored the friendship paradox from different perspectives, and proposed extensions. A few papers studied the friendship paradox from a \emph{mathematical} perspective, using concepts from graph theory to describe social networks as graphs in which the vertices represent individuals and the edges represent friendships between individuals. The goal of these papers was to \emph{quantify} the friendship paradox for several classes of sparse random graphs.  

In \cite{J}, the implications of the friendship paradox on systematic biases in perception and thought contagion are discussed. This study is based on the idea that social norms are influenced by our perceptions of the people around us. For instance, individuals are more likely to smoke when they have acquaintances who smoke \cite{CF}. In \cite{NK}, a new strategy for predicting election polls is proposed based on the friendship paradox. Moreover, research on the so-called \emph{generalised friendship paradox} investigates how attributes other than popularity can produce similar phenomena. For instance, your co-authors have more citations, more publications and more collaborators than you do \cite{EJ}, or your virtual friends receive more viral content than you do \cite{HKL}. The behaviour of the friendship paradox on random graphs, such as the Erd\H{o}s-R\'enyi random graph, has been explored in \cite{CKN} and \cite{PYNSB}. In \cite{CKN}, it is argued that the generalised friendship paradox should hold when the attribute in question correlates positively with the number of friends.

A few papers have investigated the friendship paradox from a probabilistic perspective. In \cite{CR}, for instance, it is proved that a randomly chosen friend of a randomly selected person has stochastically more friends than that person has. In \cite{HHP}, a new notion called \emph{significant friendship paradox} is introduced for locally tree-like random graphs. First, the \emph{friendship bias} of a given vertex is defined, both for simple graphs and multi-graphs, as the difference between the average degree of the neighbours of the vertex and the degree of the vertex itself. The friendship paradox is defined to be significant for a locally tree-like random graph if the proportion of vertices with non-negative friendship bias converges in probability to a number in $[\tfrac12,1]$ in the limit as the number of vertices tends to infinity. This notion is interesting because, although the friendship paradox holds in all finite graphs without self-loops, the number of vertices with negative friendship bias can vastly outnumber those with non-negative bias. Interestingly, \cite{HHP} provides mathematical and numerical evidence that the friendship paradox is significant for several well-known classes of locally tree-like random graphs, namely, the homogeneous Erd\H{o}s-R\'enyi random graph, the inhomogeneous Erd\H{o}s-R\'enyi random graph, the configuration model and the preferential attachment model. This provides valuable insight into the structure of these random graphs.

%%%

\subsection{The multi-level friendship paradox}

While \cite{SF} and the previously mentioned studies primarily addressed the friendship paradox at the individual level, subsequent research in \cite{KCR} extended the concept to higher levels. In \cite{KCR}, the concept of \emph{multi-step friendship paradox} was introduced, which considers the friendship paradox not only at the individual level but across multiple levels in the network hierarchy. It explores how the friendship paradox persists and evolves within communities, organisations, or entire populations. Later, \cite{BZ} analysed the multi-level friendship paradox in more detail from a probabilistic perspective. 

In the present paper we choose the higher levels via \emph{exploration}, i,e., we run a random walk on the graph and register the random vertex that it visits after $k$ steps. We focus on two types of exploration: backtracking random walk and non-backtracking random walk, plus a `lazy' version of backtracking random walk where the random walk has a strictly positive probability to stand still. Other choices would be possible as well, but we will not consider them. We look at the short-level and long-level behaviour of the friendship paradox in locally tree-like random graphs, where short level refers to small exploration depths and long-level refers to large exploration depths. We focus on the \emph{empirical friendship-bias distribution} $\mu_n^{(k)}$ at depth $k$ when the graph has $n$ vertices. We identify the limit of $\mu_n^{(k)}$ when $n\to\infty$ and/or $k\to\infty$. We find that for non-backtracking exploration the two limits commute, while for backtracking exploration they do not when the local limit is a subcritical tree. 

In the study of the short-level behaviour of the friendship paradox, we show that $\mu_n^{(k)}$ converges weakly to some measure $\mu_\infty^{(k)}$ as $n\to\infty$, for both types of exploration. However, to identify the weak limit of $\mu_\infty^{(k)}$ as $k\to\infty$ we must analyse the behaviour of the exploration on the limiting tree. A non-backtracking random walk on a tree can only move downward. Under the assumption that the offspring distributions are independent, we show that $\mu_\infty^{(k)}$ converges weakly to some $\mu$ as $k\to\infty$. In contrast, a backtracking random walk on a tree can move both upward and downward. We restrict our analysis to instances where the limiting tree is almost surely finite, for which it turns out that $\mu_\infty^{(k)}$ converges weakly to some $\mu^{\star}$ that is different from $\mu$.

In the study of the long-level behaviour of the friendship paradox, when the finite random graph on $n$ vertices is also almost surely connected and non-bipartite, the backtracking random walk on it is ergodic, and its stationary distribution is the limit of the $k$-step transition probabilities. Consequently, the $k$-level friendship bias converges to a stationary friendship bias as $k\to\infty$. Using this property, we establish that the weak limit of $\mu_n^{(k)}$ as $k\to\infty$ corresponds to the empirical distribution of the stationary friendship biases almost surely. Moreover, by using properties of local convergence of random graphs, we prove that the limiting law in this setting is the measure $\mu$ in probability. This result demonstrates that the limits of $\mu_n^{(k)}$ as $k\to\infty$ and $n\to\infty$ do not always commute. For the non-backtracking exploration, we model each edge of the graph as two directed edges in opposite directions and employ the non-backtracking random walk on directed edges. Notably, this approach yields the same result as the backtracking exploration. Since the assumption of almost sure connectedness may be too restrictive, we extend our results to more general random graphs, including bipartite or disconnected graphs, by incorporating a lazy exploration technique.

We also study the limit of the empirical friendship-bias distribution at a depth that depends on $n$, i.e., $k = k_n$, for choices of $k_n$ satisfying $\lim_{n\to\infty} k_n = \infty$. In the case when $k_n$ is much larger than the mixing time of the exploration we show that, under some mild additional assumptions, the limit is $\mu$ for both types of exploration. In the case when the mixing time is much smaller than the mixing time, we focus on the configuration model and non-backtracking exploration, and show that the limit remains $\mu$. Here, we make use of the possibility to couple graph explorations with certain branching processes in the configuration model. Finally, we show that, under certain uniform control conditions on the behaviour of $\mu_n^{(k)}$ with respect to both $n$ and $k$, $\mu_n^{(k_n)}$ converges weakly to the measure $\mu$ in probability for both types of exploration and any choice of $k_n$ satisfying $\lim_{n\to\infty} k_n = \infty$. 
  
The key tool in our analysis is the notion of \emph{local convergence}, which captures what a random graph looks like when viewed from a vertex that is chosen uniformly at random in the limit as $n\to\infty$. We combine this tool with properties of the local limit and properties of the exploration on the local limit in order to \emph{quantify} the multi-level friendship paradox. An important contribution of our paper is that it gives a \emph{geometric} clarification of how the bias that is inherent in the friendship paradox behaves as a function of the size of the graph and the depth of the exploration, and how it settles down to a limiting bias as both the size and the depth tend to infinity.

The detailed statements of these convergence results are given in Theorems \ref{thm1} and \ref{thm2} (for the short-level regime), Theorem \ref{thm3} (for the long-level regime), Corollary \ref{cor} and Theorem \ref{thm6} (for the lazy exploration), and Theorems \ref{thmA1}, \ref{thmA2} and \ref{thm6} (for the joint-limit regime).

%%%

\subsection{Outline}

The remainder of this paper is organised as follows. In Section \ref{sec:RWRG} we explain which random walks on random graphs we will use for the \emph{exploration} that underlies the multi-level friendship paradox. In Sections \ref{sec:shortlevel}--\ref{sec:orderoflimits} we analyse the empirical friendship-bias distribution in the limit as the level of the exploration $k$ and/or the size of the graph $n$ tend to infinity. In Section \ref{sec:shortlevel} we let $n\to\infty$ followed by $k \to \infty$ (`short level'), in Section \ref{sec:longlevel} we do the reverse (`long level'). We identify the limits and find that they can be the \emph{same} or can be \emph{different}. In Section \ref{sec:orderoflimits} we study what happens when $k$ depends on $n$, i.e., $k = k_n$ and a \emph{joint limit} is taken. Proofs of the main theorems are collected in Section \ref{sec:proofs}.

%%%%%%%%%% SECTION 2 %%%%%%%%%%%%%%%%%%%%%%%

\section{Graphs and random walks on random graphs}
\label{sec:RWRG}

Section~\ref{subsec:graphs} defines graphs and adjacency matrices. Section \ref{subsec:FPHL} describes non-backtracking exploration, Section \ref{subsec:FPHL2} backtracking exploration. Section \ref{subsec:MLFB} defines the empirical distribution of the multi-level friendship biases associated with the vertices, shows that this has non-negative mean, and identifies when the mean is zero. 

%%%

\subsection{Graphs}
\label{subsec:graphs}

Throughout this paper, a graph $G = (V(G),E(G))$ is an undirected simple graph or multi-graph. The vertices in $V(G)$ represent individuals, the edges in $E(G)$ represent mutual friendships. We use two-membered sets in $V(G)$ to represent edges, and in the presence of multiple edges $E(G)$ is treated as a multi-set. Also a self-loop is represented in $E(G)$ by a multi-set. Denote by
\begin{align*}
(d_{i}^{(G)})_{i \in V(G)}, \qquad A^{(G)} = (A_{i,j}^{(G)})_{i,j\in V(G)},
\end{align*}
the \emph{degree sequence}, respectively, the \emph{adjacency matrix} of $G$, where $A_{i,j}^{(G)}$ is the number of edges connecting $i$ and $j$. Each self-loop adds 2 to the degree. The pair $(G,o)$ will mean that $G$ is a \emph{rooted} graph with a root vertex $o$.

In the sequel, we set the empty sum equal to $0$ and the empty product equal to $1$. For $r \in \Z$, we abbreviate $\N_r = \{r,r+1,\ldots\}$. We use the standard notion for asymptotic order as $n\to\infty$: $a_n = \Theta(b_n)$ when $a_n$ is of the same order as $b_n$, $a_n = O(b_n)$ when $a_n$ is at most of order $b_n$, and $a_n = o(b_n)$ when $a_n$ is of smaller order than $b_n$. 

%%%

\subsection{Non-backtracking exploration}
\label{subsec:FPHL}

%%%
\paragraph{Random graphs.}

Let $G_{n}$ be a simple random graph with $n \in \N_{3}$ vertices labelled by $[n] = \{1,\ldots,n\}$ with minimum degree of at least $2$, defined on a probability space $(\Omega_{n},\mathcal{F}_{n},\P_{n})$. For each $n \in \N_{3}$ and $\omega\in\Omega_{n}$, we construct a non-backtracking random walk $X_{n}(\omega) = (X_{n,k}(\omega))_{k\in\N_0}$ with state space $[n]$ on a common probability space $(\Omega_{n}^{\mathrm{rw}},\mathcal{F}_{n}^{\mathrm{rw}},\P_{n}^{\mathrm{rw}})$ as follows. The initial state $X_{n,0}(\omega)$ is chosen uniformly at random from $[n]$. At each subsequent step, $X_{n,k+1}(\omega)$ is chosen uniformly at random from the neighbours of $X_{n,k}(\omega)$ in the graph $G_{n}(\omega)$, with the condition that it does not return to the state visited immediately before. 

Formally, if $P_{n}^{(k)}(i,j)$ is the probability that the non-backtracking random walk $X_n$ starting at vertex $i$ ends up at vertex $j$ after $k$ steps, then for $\omega \in \Omega_{n}$, $X_{n}(\omega) = (X_{n,k}(\omega))_{k\in\N_{0}}$ is a time-homogeneous stochastic process with initial distribution $\pi_{n,0}$ defined as
\begin{align*}
\label{initial}
\pi_{n,0}\big(\{i\}\big) = \frac{1}{n}, \qquad i\in[n],
\end{align*}
and transition kernels $(P_{n}^{(k)}(\cdot,\cdot)(\omega))_{k\in\N_0}$ with $P_{n}^{(k)}(\cdot,\cdot)(\omega)\colon\,[n]\times 2^{[n]} \rightarrow [0,1]$ defined by
\begin{align*}
P_{n}^{(k)}(i,A)(\omega) = \P_{n}^{\mathrm{rw}}\big\{X_{n,r+k}(\omega) \in A \mid X_{n,r}(\omega)=i\big\}
= \sum_{j\in A}P_{n}^{(k)}(i,j)(\omega), \qquad r\in\N_{0},
\end{align*}
where $P_{n}^{(0)}(i,j)(\omega) = \delta_{j}(\{i\})$, and $\delta_{x}$ denotes the Dirac measure at $x$. It is important to note that $(P_{n}^{(k)}(\cdot,\cdot)(\omega))_{k\in\N_0}$ does not satisfy the Chapman-Kolmogorov equation, as $X_{n}(\omega)$ is not a Markov chain.

To specify the transition probabilities more precisely, for $k\in \N$, define the set of non-backtracking $k$-walks as follows:
\begin{align*}
W_{n,k} &= \Big\{(i_0 ,i_1 ,\ldots ,i_k)\in[n]^{k+1}:\{i_l,i_{l+1}\}\in E(G_n)\text{ for all }0\leq l\leq k-1,\\
&\qquad \text{ and if } k\geq 2, \text{then also } i_{l^\prime -1}\neq i_{l^\prime +1} \text{ for all }
1\leq l^{\prime}\leq k-1\Big\}.
\end{align*}
Also, for $k\geq 2$ and $i,j\in [n]$, set
\begin{align*}
W_{n,k}(i,j) &= \Big\{(i_1 ,\ldots ,i_{k-1})\in[n]^{k-1}\colon\,(i ,i_1 ,\ldots ,i_{k-1},j) \in W_{n,k}\Big\}.
\end{align*}
Indeed, for $i,j\in [n]$,
\begin{equation}
\label{transitions1}
P_{n}(i,j) = P_{n}^{(1)}(i,j) = \frac{A_{i,j}^{(G_{n})}}{d_{i}^{(G_{n})}},
\end{equation}
and
\begin{equation}
\label{transitions2}
\begin{aligned} 
P_{n}^{(k)}(i,j) &= \sum_{(i_1 ,\ldots ,i_{k-1})\in W_{n,k}(i,j)}\dfrac{1}{\dn_{i}\,(\dn_{i_{1}}-1)\,\cdots\, (\dn_{i_{k-1}}-1)}\\
&= \sum_{(i_1 ,\ldots ,i_{k-1})\in [n]^{k-1}}\frac{A^{(G_{n})}_{i,i_{1}}}{\dn_{i}}\prod_{l=1}^{k-1}
\frac{A^{(G_{n})}_{i_{l},i_{l+1}}\mathbbm{1}_{\{i_{l-1}\neq i_{l+1}\}}}{d^{(G_{n})}_{i_{l}}-1}, \qquad k\geq 2,
\end{aligned}
\end{equation}
where we have taken $i_{0}=i$ and $i_{k}=j$ in the last line. 

To check that the transition probabilities are normalised, note that
\begin{align*}
\sum_{j \in [n]} A^{(G_{n})}_{i_{k-2},i_{k-1}}\frac{A^{(G_{n})}_{i_{k-1},j}
\mathbbm{1}_{\{i_{k-2}\neq j\}}}{d^{(G_{n})}_{i_{k-1}}-1}=A^{(G_{n})}_{i_{k-2},i_{k-1}}
\frac{\sum_{j \in [n]}A^{(G_{n})}_{i_{k-1},j}
\mathbbm{1}_{\{i_{k-2}\neq j\}}}{d^{(G_{n})}_{i_{k-1}}-A^{(G_{n})}_{i_{k-2},i_{k-1}}}
= A^{(G_{n})}_{i_{k-2},i_{k-1}}.
\end{align*}
For $k \geq 2$, this gives
\begin{align*}
\sum_{j \in [n]}P_{n}^{(k)}(i,j) = \sum_{(i_1 ,\ldots ,i_{k-1})\in [n]^{k-1}}\frac{A^{(G_{n})}_{i,i_{1}}}{\dn_{i}}
\prod_{l=1}^{k-2}\frac{A^{(G_{n})}_{i_{l},i_{l+1}}\mathbbm{1}_{\{i_{l-1}\neq i_{l+1}\}}}{d^{(G_{n})}_{i_{l}}-1},
\end{align*}
which for $k \geq 3$ equals
\begin{align*}
\sum_{(i_1 ,\ldots ,i_{k-2})\in [n]^{k-2}}\frac{A^{(G_{n})}_{i,i_{1}}}{\dn_{i}}
\prod_{l=1}^{k-3}\frac{A^{(G_{n})}_{i_{l},i_{l+1}}\mathbbm{1}_{\{i_{l-1}\neq i_{l+1}\}}}{d^{(G_{n})}_{i_{l}}-1}.
\end{align*}
Continuing this recursion, we see that
\begin{align*}
\sum_{j \in [n]} P_{n}^{(k)}(i,j) = \sum_{j \in [n]}\frac{A_{i,j}^{(G_{n})}}{d_{i}^{(G_{n})}} = 1,
\qquad k \in \N.
\end{align*}

%%%

\paragraph{Rooted random trees.} 

Consider an almost surely locally finite infinite rooted random tree $(G_{\infty},\phi)$ in which every vertex has at least one offspring, defined on a probability space $(\Omega_{\bar{\Lambda}},\mathcal{F}_{\bar{\Lambda}},\P_{\bar{\Lambda}})$, where we write $\bar{\Lambda}$ to denote the law of $(G_{\infty},\phi)$. Let $d_{j}$ be the number of \emph{offspring} of vertex $j \in V(G_{\infty})$, and $d^{(G_\infty)}_{j}$ be the \emph{degree} of $j \in V(G_{\infty})$. Note that $d^{(G_\infty)}_{\phi} = d_{\phi}$ and $d^{(G_\infty)}_{j} = d_{j}+1$ for $j\in V(G_{\infty})\setminus\{\phi\}$. For $\omega\in\Omega_{\bar{\Lambda}}$, we construct a non-backtracking random walk $X(\omega) = (X_{k}(\omega))_{k\in\N_0}$ on the tree $(G_{\infty},\phi)$ by setting $X_{0}(\omega)=\phi$ and, at level $k+1$, choosing $X_{k+1}(\omega)$ uniformly at random among the offspring of $X_{k}(\omega)$. Clearly, $X(\omega)$ is a Markov chain. For $k\in\N$, take $i_{0}=\phi$ and suppose that
\begin{align*}
I_{k} &= \Big\{(i_{1},\ldots ,i_{k})\in (V(G_{\infty}))^{k}\colon\, 
\{i_{l} ,i_{l+1}\} \in E(G_{\infty})\text{ for all }0\leq l\leq k-1,\\
&\quad\quad \text{ and if } k\geq 2, \text{then also }i_{l^\prime -1}\neq i_{l^\prime +1} \text{ for all } 1\leq l^{\prime}
\leq k-1\Big\}.
\end{align*}
Considering that $(G_{\infty},\phi)$ is an infinite tree, this implies that $I_{k}$ is non-empty for each $k\in\N$. For $j \in V(G_{\infty})$, let $P^{k}(\phi,j)$ be the probability that the non-backtracking random walk $X$ starting at the root $\phi$ ends up at vertex $j$ of the tree after $k$ steps, i.e.,
\begin{align*}
P^{k}(\phi, j) = \sum_{(i_{1},\ldots, i_{k-1}):\,(i_{1},\ldots,i_{k-1},j)\in I_{k}} \frac{1}{d_{\phi}
\prod_{l=1}^{k-1}d_{i_{l}}} = \sum_{(i_{1},\ldots, i_{k-1}):\,(i_{1},\ldots,i_{k-1},j)\in I_{k}} \frac{1}{d_{\phi}
\prod_{l=1}^{k-1}(d^{(G_\infty)}_{i_{l}}-1)}.
\end{align*}
Note that $\#\{(i_{1},\ldots, i_{k-1})\colon\,(i_{1},\ldots,i_{k-1},j)\in I_{k}\}\in\{0,1\}$.

%%%

\subsection{Backtracking exploration} 
\label{subsec:FPHL2}

%%%

\paragraph{Random graphs.} 

Let $G_{n}$ be a finite random graph without self-loops and isolated points on $n\in\N_{2}$ vertices labelled by $[n] = \{1,\ldots,n\}$, defined on a probability space $(\Omega_{n},\mathcal{F}_{n},\P_{n})$. Let us construct a backtracking random walk $X_{n}=(X_{n,k})_{k\in\N_0}$ on the graph $G_{n}$ by picking $X_{n,0}$ uniformly at random from $[n]$, and at level $k+1$ letting $X_{n,k+1}$ be chosen uniformly at random from the neighbours of $X_{n,k}$. In other words, for $\omega\in \Omega_{n}$, the stochastic process $X_{n}(\omega)=(X_{n,k}(\omega))_{k\in\N_{0}}$ is a time-homogeneous Markov chain on state space $[n]$ with initial distribution $\pi_{n,0}$ defined as
\begin{align*}
\pi_{n,0}\big(\{i\}\big) = \frac{1}{n}, \qquad i\in[n],
\end{align*}
and with transition matrix $P_{n}(\omega)=\big(P_{n}(i,j)(\omega)\big)_{i,j\in[n]}$ given by
\begin{equation}
\label{BT-transition}
P_{n}(i,j)(\omega) = \frac{A_{i,j}^{(G_{n})}(\omega)}{d_{i}^{(G_{n})}(\omega)}, \qquad i,j\in[n].
\end{equation}
We write $P_{n}^{(k)} = P_{n}^{k}$, the $k$-th power of $P_n$, to denote the $k$-step transition matrix of the backtracking random walk, which is the same notation as used for the non-backtracking random walk.  

%%%

\paragraph{Rooted random trees.} 
Let $(G_{\infty},\phi)$ be an almost surely locally finite non-trivial rooted random tree, defined on a probability space $(\Omega_{\bar{\Lambda}},\mathcal{F}_{\bar{\Lambda}},\P_{\bar{\Lambda}})$, where we use the same notations for both back-tracking and non-backtracking exploration. We construct a backtracking random walk $X=(X_{k})_{k\in\N_0}$ on the tree $(G_{\infty},\phi)$ as follows. Set $X_{0} = \phi$ and, at level $k+1$, select $X_{k+1}$ uniformly at random from the neighbours of $X_{k}$. 

For $k\in\N$, suppose that
\begin{align*}
J_{k} = \Big\{(i_{1},\ldots ,i_{k})\in (V(G_{\infty}))^{k}\colon\,\{i_{j} ,i_{j+1}\} \in E(G_{\infty}) 
\text{ for all } 0\leq j\leq k-1\Big\},
\end{align*}
where $i_{0} = \phi$. Since $(G_{\infty},\phi)$ is non-trivial, $J_{k}$ is non-empty for each $k\in\N$. For $j \in V(G_{\infty})$, let $P^{k}(\phi,j)$ be the probability that the backtracking random walk $X$ ends up at $j$ after $k$ steps, i.e.,
\begin{align*}
P^{k}(\phi, j) = \sum_{(i_{1},\ldots, i_{k-1})\colon\,(i_{1},\ldots,i_{k-1},j) \in J_{k}}  
\frac{1}{d_{\phi}\prod_{l=1}^{k-1}d^{(G_\infty)}_{i_{l}}}.
\end{align*}

%%%

\subsection{The multi-level friendship bias}
\label{subsec:MLFB}

As an extension of the concepts introduced in \cite{HHP}, we introduce the following.

\begin{definition}\label{FBdef}
{\rm (a) For $k\in\N$ and $i \in [n]$, the \emph{$k$-level friendship bias} for vertex $i$ is defined as
\begin{align*}
\Delta_{i,n}^{(k)} = \sum_{j \in [n]} P_{n}^{(k)}(i,j) \,d_{j}^{(G_{n})} - d_{i}^{(G_{n})}.
\end{align*}
(b) The $k$-level friendship bias for the root $\phi$ of the tree $(G_{\infty},\phi)$ is defined as
\begin{align*}
\Delta_{\phi}^{(k)} = \sum_{j\in V(G_{\infty})} P^{k}(\phi, j)\, d^{(G_\infty)}_{j}-d_{\phi}.
\end{align*}
Define $\mu^{(k)}_{\infty}\colon\,\mathcal{B}(\R) \to [0,1]$ to be the law of $\Delta_{\phi}^{(k)}$. Expectation with respect to $(G_{\infty},\phi)$ is denoted by $\mathbb{E}_{\bar{\Lambda}}$.}\hfill$\spadesuit$
\end{definition}

\noindent
Note that the $1$-level friendship bias $\Delta_{i,n}^{(1)}$ obtained through non-backtracking exploration and backtracking exploration are identical and coincide with the definition of $\Delta_{i,n}$ in \cite{HHP}.

\begin{definition}\label{FBdef1}
{\rm (a) The \emph{$k$-level quenched friendship bias empirical distribution} $\mu_{n}^{(k)}\colon\,\mathcal{B}(\R) \to [0,1]$ is defined as
\begin{align*}
\mu_{n}^{(k)}(\cdot) = \dfrac{1}{n}\sum_{i \in [n]} \delta_{\Delta_{i,n}^{(k)}}(\cdot),
\end{align*}
where $\mathcal{B}(\R)$ is the Borel $\sigma$-algebra on $\R$.\\ 
(b)  The \emph{$k$-level annealed friendship bias empirical distribution} $\tilde{\mu}_{n}^{(k)}\colon\,\mathcal{B}(\R) \to [0,1]$ is defined as
\begin{equation*}
\tilde{\mu}_{n}^{(k)}(\cdot) = \E_{n}\big[\mu_{n}^{(k)}(\cdot)\big],
\end{equation*}
where $\E_{n}$ denotes the expectation with respect to the random graph $G_{n}$.\\ 
(c) The \emph{average $k$-level friendship bias} is defined as
\begin{equation*}
\Delta_{[n]}^{(k)} = \frac{1}{n} \sum_{i \in [n]} \Delta_{i,n}^{(k)} = \int_{\R} x\,\mu_{n}^{(k)}(\dd x), 
\end{equation*}
with 
\begin{align*}
\E_{n}\big[\Delta_{[n]}^{(k)}\big] = \frac{1}{n}\sum_{i \in [n]} \E_{n}\big[\Delta_{i,n}^{(k)}\big] 
= \int_{\R} x\,\tilde{\mu}_{n}^{(k)}(\dd x).
\end{align*}
}\hfill$\spadesuit$
\end{definition}

\noindent
Note that, for fixed $\omega\in \Omega_{n}$, 
\begin{align*}
\Delta_{i,n}^{(k)}(\omega) = \E_{n}^{\mathrm{rw}}\Big[d_{X_{n,k}(\omega)}^{(G_{n}(\omega))}
-d_{X_{n,0}(\omega)}^{(G_{n}(\omega))}\mid X_{n,0}(\omega)=i\Big],
\qquad
\Delta_{[n]}^{(k)}(\omega) = \E_{n}^{\mathrm{rw}}\Big[d_{X_{n,k}(\omega)}^{(G_{n}(\omega))}
-d_{X_{n,0}(\omega)}^{(G_{n}(\omega))}\Big],
\end{align*} 
where $\E_{n}^{\mathrm{rw}}$ denotes expectation with respect to the random walk $X_{n}(\omega)$. 

We refer to the property that $\Delta_{[n]}^{(k)} \geq 0$ as the \emph{$k$-level friendship paradox}: in a social network with $n$ individuals on average the $k$-level friends of an individual have at least as many friends as the individual itself. The following theorem establishes this property and identifies when equality holds. 

\begin{theorem}
\label{thm0}
For each $k\in\N$, $\Delta_{[n]}^{(k)} \geq 0$, where for:
\begin{itemize}
\item[{\rm (1)}] 
non-backtracking exploration equality holds if and only if the degrees of the endpoints of each non-backtracking walk of length $k$ are equal;
\item[{\rm (2)}] 
backtracking exploration equality holds if and only if the degrees of the endpoints of each backtracking walk of length $k$ are equal. 
\end{itemize}
The former holds for all $k$ if each connected component of $G_n$ is regular. The latter holds for odd $k$ if and only if each connected component of $G_n$ is regular, and for even $k$ if and only if each connected component of $G_n$ is either regular or bi-regular bipartite.
\end{theorem}

Note that the equality condition of Theorem \ref{thm0} refers to each individual walk of length $k$, i.e., for every fixed walk of length $k$, the degrees of its two endpoints are equal, although this degree may differ between distinct walks.

\begin{remark}
\rm{In Theorem \ref{thm0}, the converse part of the characterization for non-backtracking exploration does not generally hold. For instance, in realisation (a) of Figure \ref{fig0}, although the degrees of the endpoints of every non-backtracking walk of length $3$ are equal, the corresponding connected component is not regular. Similarly, in realisation (b) of Figure \ref{fig0}, while the degrees of the endpoints of each non-backtracking walk of length $4$ are the same, the connected component is not regular.}\hfill$\spadesuit$

%%%%%%%%%%%%%%%%%%%%%%%%%%%%%%%%%%%%%%%%%%%%
\vspace*{4mm}
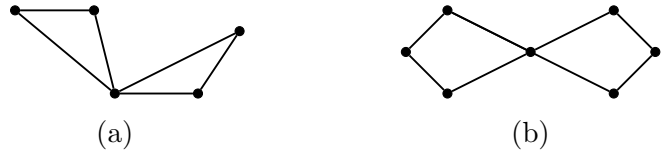
\begin{figure}[htb]
\centering
\begin{tikzpicture}[scale=1.1]
%%Fig(a)
\node[draw, circle, fill=black, inner sep=1.2pt] (A) at (-1.2,1) {};
\node[draw, circle, fill=black, inner sep=1.2pt] (B) at (0,0) {};
\node[draw, circle, fill=black, inner sep=1.2pt] (C) at (-0.25,1) {};
\node[draw, circle, fill=black, inner sep=1.2pt] (D) at (1,0) {};
\node[draw, circle, fill=black, inner sep=1.2pt] (E) at (1.5,0.75) {};
\draw[-,line width=0.25mm] (A) -- (B) -- (C) -- (A);
\draw[-,line width=0.25mm] (B) -- (D);
\draw[-,line width=0.25mm] (B) -- (E) -- (D);
\node at (0,-0.5) {(a)};
%%Fig(b)
\node[draw, circle, fill=black, inner sep=1.2pt] (F) at (3.5,0.5) {};
\node[draw, circle, fill=black, inner sep=1.2pt] (G) at (4,1) {};
\node[draw, circle, fill=black, inner sep=1.2pt] (H) at (5,0.5) {};
\node[draw, circle, fill=black, inner sep=1.2pt] (I) at (4,0) {};
\node[draw, circle, fill=black, inner sep=1.2pt] (J) at (6,1) {};
\node[draw, circle, fill=black, inner sep=1.2pt] (K) at (6.5,0.5) {};
\node[draw, circle, fill=black, inner sep=1.2pt] (L) at (6,0) {};
\draw[-,line width=0.25mm] (F) -- (G) -- (H) -- (I) -- (F);
\draw[-,line width=0.25mm] (H) -- (J) -- (K) -- (L) -- (G);
\node at (5,-0.5) {(b)};
\end{tikzpicture}
\caption{\small Two examples of non-regular connected components for which equality in the $k$-level friendship paradox holds for $k=3$, respectively, $k=4$.}
\label{fig0}
\end{figure}
%%%%%%%%%%%%%%%%%%%%%%%%%%%%%%%%%%%%%%%%%%%%
\end{remark}

%%%%% SECTION 3 %%%%%%%%%%%%%%%%%%%%%%%

\section{Short-level friendship bias for large sparse graphs}
\label{sec:shortlevel}

The asymptotic behaviour of the $k$-level quenched friendship bias empirical distribution $\mu_{n}^{(k)}$ as $n\to\infty$ for fixed $k$ provides insights into the $k$-level friendship paradox within a large network. Theorem \ref{thm1} below investigates the relationship between $\mu_{n}^{(k)},\tilde{\mu}_{n}^{(k)}$ and $\mu^{(k)}_{\infty}$ for locally tree-like random graphs.

We briefly recall the notion of local convergence for random graphs, following \cite[Section 2]{HHP}. Additional details can be found in \cite[Chapter 2]{RvdH2}.

For a rooted graph $(H,o)$, let $B_{r}^{(H)}(o) = ((V(B_{r}^{(H)}(o)),E(B_{r}^{(H)}(o))),o)$ be a rooted subgraph with root vertex $o$, defined as
\begin{equation*}
\begin{aligned}
V(B_{r}^{(H)}(o)) &= \big\{i\in V(H)\colon\,\mathrm{dist}_{H}(o,i)\leq r\big\},\\
E(B_{r}^{(H)}(o)) &= \big\{\{i,j\}\in E(H)\colon\,\max\{\mathrm{dist}_{H}(o,i),\mathrm{dist}_{H}(o,j)\}\leq r\big\},
\end{aligned}
\end{equation*}
where $\mathrm{dist}_{H}$ represents the graph distance in $H$. We write $H_1 \simeq H_2$ to denote that the two graphs $H_1$ and $H_2$ are \emph{isomorphic}. 

Let $\mathscr{G}$ be the set of all connected locally finite rooted graphs equipped with the metric
\begin{align*}
\mathrm{d}_{\mathscr{G}}\big((H_{1},o_{1}),(H_{2},o_{2})\big)
	= \frac{1}{1+\sup\{r\geq 0\colon\,B_{r}^{(H_1)}(o_1) \simeq B_{r}^{(H_2)}(o_2)\}}.
\end{align*}
Note that $\mathrm{d}_{\mathscr{G}}$ can be considered a metric, following the convention that two connected locally finite rooted graphs $(H_{1},o_{1})$ and $(H_{2},o_{2})$ are identified when $(H_{1},o_{1}) \simeq (H_{2},o_{2})$. 

\begin{definition}
\label{deflc}
{\rm Let $(H_{n})_{n\in\N}$ be a sequence of finite random graphs. 
For each $n\in\N$, associate to $H_{n}$ the rooted graph $\mathscr{C}(H_{n},U_{n})$, defined as the connected component containing a uniformly chosen vertex $U_{n}\in V(H_{n})$, which serves as the root.
\begin{itemize}
\item[(a)] 
$H_{n}$ converges \emph{locally weakly} to $(H,o)\in \mathscr{G}$ with law $\tilde{\nu}$ if, for every bounded and continuous function $h\colon\,\mathscr{G}\to\R$,
\begin{align*}
	\E_{\tilde{\nu}_{n}}\big[h(\mathscr{C}(H_{n} ,U_{n}))\big] \to \E_{\tilde{\nu}}\big[h((H,o))\big],
\end{align*}
where $\E_{\tilde{\nu}_{n}}$ is the expectation with respect to the random vertex $U_{n}$ and the random graph $H_{n}$ with joint law $\tilde{\nu}_{n}$, while $\E_{\tilde{\nu}}$ is the expectation with respect to $(H, o)$ with law $\tilde{\nu}$.
\item[(b)] 
$H_{n}$ converges \emph{locally in probability} to $(H,o)\in \mathscr{G}$ with law $\tilde{\nu}$ if, for every bounded and continuous function $h\colon\,\mathscr{G}\to\R$,
\begin{align*}
	\E_{\tilde{\nu}_{n}}\big[h(\mathscr{C}(H_{n} ,U_{n})) \mid H_{n}\big] \prob\E_{\tilde{\nu}}\big[h((H,o))\big].
\end{align*}
\end{itemize}
}\hfill$\spadesuit$
\end{definition}

\noindent
Note that in Definition~\ref{deflc} the law $\tilde{\nu}$ of the limiting rooted tree $(H,o)$ is necessarily deterministic when the convergence is locally weak. In contrast, when the convergence is locally in probability $\tilde{\nu}$ may depend on the realisation of $(H_{n})_{n\in\N}$, and hence the limiting law may itself be random, for example, a mixture of random environment settings.

\begin{theorem}
\label{thm1}
For both types of exploration and each $k\in\N$:
\begin{itemize}
\item[\rm{(a)}] 
If $G_{n}$ converges locally in probability to $(G_{\infty},\phi)$ as $n\to\infty$, then $\mu_{n}^{(k)} \weak \mu^{(k)}_{\infty}$ as $n\to\infty$ in probability. In particular, as $n\to\infty$,
\begin{align*}
\mu_{n}^{(k)} (A) \prob \mu^{(k)}_{\infty} (A), \qquad 
\tilde{\mu}_{n}^{(k)} (A) \to \mu^{(k)}_{\infty} (A) \qquad \forall\,A \in \mathcal{B}(\R).
\end{align*}
\item[\rm{(b)}] 
If $G_{n}$ converges locally weakly to $(G_{\infty},\phi)$ as $n\to\infty$, then $\tilde{\mu}_{n}^{(k)} \weak \mu^{(k)}_{\infty}$ as $n\to\infty$. In particular, as $n\to\infty$,
\begin{align*}
\tilde{\mu}_{n}^{(k)} (A) \to \mu^{(k)}_{\infty} (A) \qquad \forall\,A \in \mathcal{B}(\R).
\end{align*}
\end{itemize}
\end{theorem}

\noindent
Theorem \ref{thm1} shows that, for locally tree-like random graphs satisfying the mentioned conditions, we have $\mu_{n}^{(k)}([0,\infty)) \prob \mu^{(k)}_{\infty}([0,\infty))$ for each $k\in\N$. 

\begin{remark}
\rm{A result similar to Theorem~\ref{thm1} also holds for more general connected and almost surely locally finite rooted random graphs (not necessarily rooted random trees), which provides an extension of \cite[Theorems~2.2 and 2.3]{HHP} to arbitrary fixed $k \in \N$. This extension requires additional assumptions on the graphs (similar to those outlined in Section~\ref{sec:RWRG} for the tree setting) to ensure that the $k$-step exploration is well defined. However, the primary focus of the present paper is on rooted random trees.}
\end{remark}

Theorem \ref{thm2} below identifies the limit of $\mu^{(k)}_{\infty}$ as $k\to\infty$. We label the vertices in $V(G_{\infty})\setminus\{\phi\}$ by natural numbers.

\begin{theorem}
\label{thm2}
$\mbox{}$
\begin{itemize}
\item[\rm{(a)}] 
Consider the non-backtracking exploration. 
Let $G_n$ converge locally in probability to $(G_\infty,\phi)$ as $n\to\infty$, and let the minimum degree in $G_n$ be at least $3$ for all sufficiently large $n$. Also, let
$(G_{\infty},\phi)$ have a deterministic law $\bar{\Lambda}$ and satisfy the following conditions:
\begin{itemize}
\item[{\rm (I)}] 
$d_{\phi}$ is independent of $(d_{i})_{i\in\N}$.
\item[{\rm (II)}] 
$(d_{i})_{i\in\N}$ are i.i.d.\ with a finite second moment.
\end{itemize}
%Also let the minimum degree in the graph $G_n$ is at least $3$. 
Under these conditions, as $k\to\infty$,
\begin{align*}
\Delta_{\phi}^{(k)} \prob \Delta_{\phi} = \frac{m^{(2)}}{m^{(1)}}-d_{\phi}
\end{align*}
with $m^{(1)} = \mathbb{E}_{\bar{\Lambda}}[d_{\phi}]$ and $m^{(2)} = \mathbb{E}_{\bar{\Lambda}}[d_{\phi}^{2}]$. In particular, $\mu^{(k)}_{\infty} \weak \mu$ as $k\to\infty$, where 
\begin{equation}
\label{eq:mudef}
\mu \text{ is the law of } \Delta_{\phi}.
\end{equation}
\item[\rm{(b)}] 
Consider the backtracking exploration. Let $(G_{\infty},\phi)$ be almost surely finite and non-bipartite, with law $\bar{\Lambda}$. Then, as $k\to\infty$,
\begin{align*}
\Delta_{\phi}^{(k)} \as \Delta_{\phi}^{\star} = \frac{\sum_{j\in V(G_{\infty})}(d_{j}^{(G_{\infty})})^2}{\sum_{j\in V(G_{\infty})}d_{j}^{(G_{\infty})}}-d_{\phi}.
\end{align*}
In particular, if $\bar{\Lambda}$ is deterministic, then $\mu^{(k)}_{\infty} \weak \mu^{\star}$ as $k\to\infty$, where 
\begin{equation}
\label{eq:mustardef}
\mu^{\star} \text{ is the law of } \Delta_{\phi}^{\star}.
\end{equation}
\end{itemize}
\end{theorem}

We will see in Lemma \ref{lem0} that, under the conditions of Theorem \ref{thm2}(a), the distribution of $d_1$ is the \emph{size-biased} version of the distribution of $d_\phi$. Hence, a finite first moment of $d_1$ implies a finite second moment of $d_\phi$.

\begin{remark}
{\rm (1) A general form of Theorem \ref{thm2}(a) can be formulated for both the deterministic and the random case of the measure $\bar{\Lambda}$, where the limit is given by $\Delta_{\phi} =\mathbb{E}_{\bar{\Lambda}}[d_{1}]+1-d_{\phi}$. When $\bar{\Lambda}$ is deterministic, this aligns with the classic notion of unimodularity of $\bar{\Lambda}$, whereby $\mathbb{E}_{\bar{\Lambda}}[d_{1}]+1=\frac{m^{(2)}}{m^{(1)}}$ (see Lemma \ref{lem0}). Since for most natural locally tree-like random graphs $\bar{\Lambda}$ is deterministic, we will focus on this case as a natural choice. For the notion of unimodular random measured metric spaces, we refer the reader to \cite{Kh}.\\
(2) Theorem \ref{thm2} shows that, interestingly, $\mu \neq \mu^{\star}$, i.e., non-backtracking differs from backtracking. The assumption in Theorem \ref{thm2}(b) that the random tree be almost surely finite is restrictive, because it means the local limit is a subcritical tree. We do not know whether this assumption can be dropped.}\hfill$\spadesuit$
\end{remark}

%%%%% SECTION 4 %%%%%%%%%%%%%%%%%%%%%%%

\section{Long-level friendship bias for large sparse graphs}
\label{sec:longlevel}

In Section~\ref{subsec:convexpl} we show that both explorations converge to an equilibrium. In Section~\ref{subsec:ll} we identify the limit of $\mu^{(k)}_n$ as $k \to \infty$ both for $n$ fixed and followed by $n\to\infty$. In Section~\ref{subsec:lazy} we look at lazy exploration that can stay put with a positive probability. 

%%%

\subsection{Convergence of the explorations}
\label{subsec:convexpl}

Consider the backtracking exploration and $G_{n}$ as defined in Section \ref{subsec:FPHL2}. For $\omega\in \Omega_{n}$, let $G_{n}(\omega)$ be connected. Since $G_{n}(\omega)$ is connected and finite, the random walk $X_{n}(\omega)$ is a positive recurrent irreducible homogeneous Markov chain. Therefore it has a unique stationary distribution \cite[Theorem 3.2.6]{B}. The stationary distribution, denoted by $\pi_{n}(\cdot)(\omega)$, can be derived from the global balance equation as follows:
\begin{align}\label{statform}
\pi_{n}(i)(\omega) = \pi_{n}(\{i\})(\omega)=\dfrac{\dn_{i}(\omega)}{\sum_{j\in [n]}\dn_{j}(\omega)}, \qquad i\in [n].
\end{align}
The random walk $X_{n}(\omega)$ is aperiodic if and only if $G_{n}(\omega)$ is non-bipartite. Thus, if we also assume that $G_{n}(\omega)$ is non-bipartite, then $X_{n}(\omega)$ is ergodic; in particular,
\begin{align}
\label{statlim}
\lim_{k\rightarrow\infty}P_{n}^{(k)}(i,j)(\omega)=\pi_{n}(j)(\omega)
\end{align}
for all $i,j\in[n]$ \cite[Theorem 1.8.3]{N}. 

Next, consider the non-backtracking exploration and $G_{n}$ as defined in Section \ref{subsec:FPHL}. Here, we employ the non-backtracking random walks on the directed edges, as in \cite{K}, to establish that the limit in \eqref{statlim} holds in this context as well.
Let $\overset{\rightarrow}{\{i,j\}}$ denote a directed edge from $i$ to $j$, where the `order' of the vertices is important. Define the set of directed edges by
\begin{align*}
\vec{E}(G_{n}) = \Big\{\overset{\rightarrow}{\{i,j\}}: \{i,j\}\in E(G_{n})\Big\}.
\end{align*}
Indeed $|\vec{E}(G_{n})|=2|E(G_{n})|=\sum_{j\in [n]}\dn_{j}$. For $\omega\in \Omega_{n}$, define a matrix 
\[\vec{P}_{n}(\omega)=\Big(\vec{P}_{n}\big(\{\overset{\rightarrow}{i_{1},i_{2}}\},\{\overset{\rightarrow}{j_{1},j_{2}}\}\big)(\omega)\Big)_{\{\overset{\rightarrow}{i_{1},i_{2}}\},\{\overset{\rightarrow}{j_{1},j_{2}}\}\in\vec{E}(G_{n}(\omega))}\]
by
\begin{align*}
\vec{P}_{n}\big(\{\overset{\rightarrow}{i_{1},i_{2}}\},\{\overset{\rightarrow}{j_{1},j_{2}}\}\big)(\omega) = \left\{
\begin{array}{ll}
\frac{1}{\dn_{i_2}(\omega)-1}, & \text{if }i_{2}=j_{1}\text{ and }i_{1}\neq j_{2}, \\
0, & \text{otherwise}.
\end{array}
\right.
\end{align*}
Note that $\vec{P}_{n}(\omega)$ is a one-step transition probability matrix for a random walk that operates as a Markov chain on the directed edges of the directed version of $G_{n}(\omega)$, in which each edge has been replaced by two directed edges, one in each direction. It can be easily proved that $\vec{P}_{n}(\omega)$ is doubly stochastic as for each $\{\overset{\rightarrow}{j_{1},j_{2}}\}\in\vec{E}(G_{n}(\omega))$,
\begin{align*}
\sum_{\{\overset{\rightarrow}{i_{1},i_{2}}\}\in\vec{E}(G_{n}(\omega))}\vec{P}_{n}\big(\{\overset{\rightarrow}{i_{1},i_{2}}\},\{\overset{\rightarrow}{j_{1},j_{2}}\}\big)(\omega) 
=\sum_{\substack{i_1 \in [n] \\ \{i_1, j_1\} \in E(G_n(\omega)) \\ i_{1}\neq j_{2}}} \frac{1}{\dn_{j_1}(\omega)-1}=1.
\end{align*}
Therefore, the uniform distribution on $\vec{E}(G_{n})(\omega)$ is a stationary distribution for the corresponding random walk. If $\vec{P}_{n}(\omega)$ is irreducible and aperiodic, then for each $\{\overset{\rightarrow}{i_{1},i_{2}}\},\{\overset{\rightarrow}{j_{1},j_{2}}\}\in\vec{E}(G_{n}(\omega))$,
\begin{align*}
\vec{P}_{n}^{k}\big(\{\overset{\rightarrow}{i_{1},i_{2}}\},\{\overset{\rightarrow}{j_{1},j_{2}}\}\big)(\omega) \to \dfrac{1}{\sum_{j\in [n]}\dn_{j}(\omega)},\quad\quad k\to\infty ,
\end{align*}
 where $\vec{P}_{n}^{k}(\omega)$ is the $k$th power of the matrix $\vec{P}_{n}(\omega)$ \cite[Theorem 1.8.3]{N}. 
In this way, since for every $i_{1},j_{1}\in [n]$,
 \begin{align*}
 \sum_{\substack{i_{2} \in [n] \\ \{i_{1},i_{2}\} \in E(G_n(\omega))}}\sum_{\substack{j_{2} \in [n] \\\{j_{1},j_{2}\}\in E(G_{n}(\omega))}}\vec{P}_{n}^{k}\big(\{\overset{\rightarrow}{i_{1},i_{2}}\},\{\overset{\rightarrow}{j_{1},j_{2}}\}\big)(\omega) =\dn_{i_{1}}(\omega)P_{n}^{(k)}(i_{1},j_{1})(\omega),
 \end{align*}
we can conclude \eqref{statlim} for all $i,j\in[n]$. 

%%%

\subsection{Long-level exploration}
\label{subsec:ll} 

Theorem \ref{thm3} below identifies the limit of $\mu_{n}^{(k)}$ as $k\to\infty$ and specifies the behaviour of this limit as $n \to \infty$.

\begin{theorem}
\label{thm3}
For backtracking exploration, assume that $G_n$ is almost surely connected and non-bipartite. For non-backtracking exploration, assume that $\vec{P}_{n}$ is almost surely irreducible and aperiodic. Then the following hold:
\begin{itemize}
\item[{\rm (a)}] 
For fixed $n$, $\mu_{n}^{(k)} \weak \mu_{n}^{(\infty)}$ as $k\to\infty$ almost surely, where 
\begin{align*}
\mu_{n}^{(\infty)}(\cdot)= \dfrac{1}{n}\sum_{i \in [n]} \delta_{\Dst_{i,n}}(\cdot)
\end{align*}
with
\begin{align}
\label{Deltast}
\Dst_{i,n}= \sum_{j\in [n]}\pi_{n}(j)\,\dn_{j}-\dn_{i}
\end{align}
the stationary friendship bias on $G_n$. 
\item[{\rm (b)}]
Suppose that $((\dn_{i})^{2})_{i\in [n],\,n\in\N_{3}}$ and $((\dn_{i})^{2})_{i\in [n],\,n\in\N_{2}}$ are uniformly integrable for non-backtracking exploration, respectively, backtracking exploration. If $G_n$ converges locally in probability to $(G_{\infty},\phi)$ as $n\to\infty$, with a deterministic law $\bar{\Lambda}$, then $\mu_n^{(\infty)} \weak \mu$ in probability as $n\to\infty$, with $\mu$ the same distribution as in \eqref{eq:mudef}.
\end{itemize}
\end{theorem}

\begin{remark}
\rm{The irreducibility and aperiodicity of the non-backtracking transition matrix $\vec{P}_n$ follow from the structural properties of the underlying graph $G_n$ under some additional assumptions on the degrees. If a graph is finite, connected, has minimum degree at least $2$ and has at least one vertex of degree exceeding $2$, then the directed edge graph associated with it is strongly connected, so that the corresponding non-backtracking transition matrix is irreducible (see \cite[Claim~4]{EH}). The non-backtracking Markov chain on a finite graph with minimum degree at least $2$ is aperiodic whenever the underlying graph is non-bipartite and is not a single cycle (i.e., it has maximum degree exceeding $2$), since the presence of an odd cycle allows return paths of both even and odd lengths. Hence, for finite connected non-bipartite graphs with a minimum degree at least $2$ and maximum degree exceeding $2$, the non-backtracking Markov chain on directed edges is both irreducible and aperiodic.}
\end{remark}

%%% 

\subsection{Lazy exploration}
\label{subsec:lazy}
Non-bipartiteness is a mild condition, but we can even drop it when we consider `lazy exploration'. More precisely, fix $\omega\in\Omega_{n}$ and consider a random walk on $G_{n}(\omega)$ with a laziness factor $\delta \in (0,1)$. The walk starts from a vertex chosen uniformly at random and, at each step, with probability $\delta$ remains at its current position and with probability $1-\delta$ moves to a neighbouring vertex chosen uniformly at random. The distribution $\pi_{n}(\omega)$ defined in \eqref{statform} again serves as the stationary distribution for the lazy random walk as well. Since the lazy random walk is aperiodic, if the graph $G_{n}(\omega)$ is connected, then the lazy random walk is ergodic. In particular, if $P_{n}^{(k,\mathrm{lazy})}(\omega)$ is the $k$-step transition matrix of the lazy random walk, then
\begin{align*} 
\lim_{k \rightarrow \infty} P_{n}^{(k,\mathrm{lazy})}(i,j)(\omega) 
= \lim_{k \rightarrow \infty} \Big(\delta I_{n} + (1 - \delta) P_{n}(\omega)\Big)^{k}(i,j) = \pi_{n}(j)(\omega) \qquad \forall\,i, j \in [n],
\end{align*}
where $I_{n}$ is the $n \times n$ identity matrix and $P_{n}(\omega)$ is the one-step transition matrix of the backtracking random walk on $G_{n}(\omega)$ \cite[Theorem 1.8.3]{N}. We add the label `lazy' to indicate the lazy exploration. 

The following corollary immediately follows from the proof of Theorem \ref{thm3}.

\begin{corollary} 
\label{cor}
Consider lazy exploration, and assume that $G_n$ is almost surely connected. Then, for any fixed $n$, $\mu_{n}^{(k,\mathrm{lazy})} \weak \mu_{n}^{(\infty)}$ as $k \to \infty$ almost surely, where $\mu_{n}^{(\infty)}$ is as defined in Theorem \ref{thm3}. Moreover, if $((\dn_{i})^{2})_{i \in [n], n \in \N_2}$ is uniformly integrable and $G_n$ converges locally in probability as $n\to\infty$ to $(G_{\infty}, \phi)$ with a deterministic law $\bar{\Lambda}$, then $\mu_n^{(\infty)} \weak \mu$ as $n \to \infty$ in probability, where $\mu$ is the distribution defined in \eqref{eq:mudef}. 
\end{corollary}

Unlike the non-bipartiteness condition, the connectedness condition is a significant restriction, as many interesting random graphs are not connected almost surely. In these cases, we can look at each connected component separately. However, assuming non-bipartiteness for all the connected components is a rather severe restriction. A practical solution is to focus on the lazy exploration on each connected component. A lazy random walk on each connected component $C_{i}=C_{i}(G_{n})$ containing a vertex $i\in V(G_{n})$ is ergodic with a (unique) stationary distribution given by 
\begin{align*}
\pi_{n}^{(\mathrm{c})}(i) = \pi_{n}^{(\mathrm{c})}(\{i\})=\dfrac{\dn_{i}}{\sum_{j\in V(C_{i})}\dn_{j}}, \qquad i\in [n].
\end{align*}
Specifically, for all $i\in [n]$ and $j\in V(C_{i})$,
\begin{align}
\label{statlim2}
\lim_{k \rightarrow \infty} P_{n}^{(k,\mathrm{lazy})}(i,j)=\pi_{n}^{(\mathrm{c})}(j),
\end{align}
which holds everywhere \cite[Theorem 1.8.3]{N}. 

\begin{theorem}
\label{thm6}
For lazy exploration the following hold:
\begin{itemize}
\item[{\rm (a)}] 
For fixed $n\in\N_{2}$, $\mu_{n}^{(k,\mathrm{lazy})} \weak \mu_{n}^{(\infty ,\mathrm{c})}$ as $k\to\infty$ almost surely, where 
\begin{align*}
\mu_{n}^{(\infty ,\mathrm{c})}(\cdot)= \dfrac{1}{n}\sum_{i \in [n]} \delta_{\Delta^{(\mathrm{st},\mathrm{c})}_{i,n}}(\cdot)
\end{align*}
with
\begin{align*}
\Delta^{(\mathrm{st},\mathrm{c})}_{i,n}= \sum_{j\in [n]}\pi_{n}^{(\mathrm{c})}(j)\,\dn_{j}-\dn_{i}.
\end{align*}
\item[{\rm (b)}]
Suppose that $((\dn_{i})^{2})_{i\in [n],\,n\in\N_{2}}$ is uniformly integrable, and $\P_{n}\{G_{n}\text{ is connected}\} \to 1$ as $n\to\infty$. If $G_n$ converges locally in probability as $n\to\infty$ to $(G_{\infty},\phi)$, with a deterministic law $\bar{\Lambda}$, then $\mu_n^{(\infty ,\mathrm{c})} \weak \mu$ in probability as $n\to\infty$, with $\mu$ the same distribution as in \eqref{eq:mudef}.
\end{itemize}
\end{theorem}

%%%%%% SECTION 5 %%%%%%%%%%%%%%%%%%%%%%%%%%%

\section{Role of the order of the limits of graph size and exploration depth}
\label{sec:orderoflimits}

Theorems \ref{thm1}, \ref{thm2} and \ref{thm3} suggest that the choice of exploration does not matter for the multi-level friendship paradox in the limit as $k,n\to\infty$. Moreover, Theorems \ref{thm2} and \ref{thm3} show that the limit of $\mu_{n}^{(k)}$ as $k,n \to \infty$ does not depend on the order in which the limits are taken, at least for the non-backtracking exploration, provided the conditions are met. 

In \cite{HHP}, we looked at four classes of random graphs: the homogeneous Erd\H{o}s-R\'enyi random graph (HER), the inhomogeneous Erd\H{o}s-R\'enyi random graph (IER), the Configuration Model (CM), the Preferential Attachment Model (PAM). For all four classes the local limit is \emph{known} and is given by an interesting \emph{random rooted tree} (see \cite{RvdH1,RvdH2}). All four classes satisfy the conditions in Theorem \ref{thm6} for $n$ large enough, provided we restrict to the giant component for (HER) and (IER) to ensure connectedness. Only the first three classes satisfy the conditions in Theorem \ref{thm2} for $n$ large enough. Indeed, for (PAM) the local limit is a P\'olya point tree, for which conditions (I) and (II) fail. For (HER) the limit is a \emph{Galton-Watson tree} with offspring distribution {\rm Poisson}($\lambda$), for (IER) the limit is a \emph{unimodular multi-type marked Galton-Watson tree}, while for (CM) the limit is a \emph{unimodular branching process tree} (see \cite{RvdH2} for definitions and proofs), all of which satisfy conditions (I) and (II). 

%%%

\subsection{Two cases: after and before mixing}

\begin{definition}
{\rm We call a diverging sequence of times $(\psi_{n})_{n\in\N}$ the \textit{mixing time scale} when for any sequence $(k_{n})_{n\in\N}$ in $\N_{0}$ the following conditions hold:}
\begin{itemize}
\item[{\rm (1)}] {\rm If $\lim_{n\to\infty} k_{n}/\psi_{n}=0$, then $\lim_{n\to\infty} \mathcal{D}_{n}(k_{n}) = 1$ in probability.}
\item[{\rm (2)}] {\rm If $\lim_{n\to\infty} k_{n}/\psi_{n}=\infty$, then $\lim_{n\to\infty} \mathcal{D}_{n}(k_{n}) = 0$ in probability.}
\end{itemize}
{\rm Here, $\mathcal{D}_{n}(k)$ is the worst case total variation distance to the stationary distribution on $G_n$ of the exploration on $G_n$ at time $k$.}
\hfill$\spadesuit$
\end{definition}

\begin{theorem}
\label{thmA1}
For both types of exploration, let $(\psi_{n})_{n\in\N}$ be the mixing time scale of the random walk $(X_{n})_{n\in\N}$. Suppose that $((\dn_{i})^{2})_{i\in [n],\,n\in\N_{3}}$ and $((\dn_{i})^{2})_{i\in [n],\,n\in\N_{2}}$ are uniformly bounded for non-backtracking exploration, respectively, backtracking exploration. If $G_n$ converges locally in probability to $(G_{\infty},\phi)$ as $n\to\infty$, with a deterministic law $\bar{\Lambda}$, then $\mu_{n}^{(k_n)} \weak \mu$ as $n\to\infty$ in probability for any $k_n$ such that $\lim_{n\to\infty} k_{n}/\psi_{n}=\infty$, and $\mu$ is the same distribution as in \eqref{eq:mudef}.
\end{theorem}

Let $\CM_{n}(\mathbf{d}_{n})$ denote the configuration model on $[n]$ with a degree sequence $\mathbf{d}_{n}$. Suppose that $d_{\max}$ denotes the maximum vertex degree in the degree sequence $\mathbf{d}_{n}$. 
Let $D_{n}$ be the degree of a uniformly chosen vertex in $[n]$. If $D_{n}$ converges in distribution to a random variable $D$ with $\P\{D\geq 1\}=1$ such that $\E[D_{n}] \to \E[D]<\infty$, then this random graph converges locally in probability to a unimodular branching process tree $(G_{\infty},\phi)$ with root offspring distribution $p=(p_{k})_{k\in \N_0}$ given by $p_{k} = \P\{D=k\}$ and with offspring distribution $p^{\star}=(p^{\star}_{k})_{k\in \N_0}$ for all other vertices, where $p^{\star}$ is the size-biased version of $p$ given by
\begin{align}
\label{CM-neq--2}
p^{\star}_{k} =\P\{D^{\star}-1=k\} = \frac{(k+1)p_{k+1}}{\E[D]}
\end{align}
\cite[Theorem 4.1 and Definition 1.26]{RvdH2}. Take $m_{n}\in\N$ such that
\begin{align}
\label{CM-neq--1}
m_{n}=o\left(\sqrt{\frac{n}{d_{\max}}}\,\right)\quad \text{with}\quad d_{\max}=o(n),\qquad n\to \infty.
\end{align}

\begin{theorem}
\label{thmA2}
Consider the non-backtracking exploration. Suppose that $G_{n}=\CM_{n}(\mathbf{d}_{n})$ is such that $\inf_{n\in\N}\inf_{i\in[n]} d_{i}^{(G_{n})} \geq 3$ and $\sup_{n\in\N}\sup_{i\in[n]}d_{i}^{(G_{n})} = M < \infty$. If $D_n\weak D$ with $\P\{D\geq 1\}=1$, then $\mu_{n}^{(k_n)} \weak \mu$ as $n\to\infty$ in distribution for any $k_{n}\leq \frac{\log m_{n}}{2\log M}$ with $k_{n}\to\infty$ as $n\to\infty$, where $\mu$ is the law of
\[
\Delta_{\phi}= \frac{\E[D^2]}{\E[D]} -d_{\phi},
\]
and $m_n$ is defined as in \eqref{CM-neq--1}. 
\end{theorem}

\noindent
By picking $m_n = (n/\log n)^{1/2}$, we get $\frac{\log m_{n}}{2\log M} \asymp \log n$, and so the result in Theorem~\ref{thmA2} holds throughout the pre-mixing regime when the mixing time satisfies $\psi_n \asymp \log n$. The latter holds for the configuration model with bounded degrees \cite{BenHS}.

When we view the random walk as moving along edges rather than between pairs of vertices, we can define a non-backtracking random walk on a non-simple graph as a random walk that cannot backtrack along the edge it just crossed, but can backtrack along any another edge between the same pair of vertices. In Theorem \ref{thmA2} we consider a non-backtracking random walk on the vertices of the finite exploration of the configuration model. For this purpose, we run the finite non-backtracking random walk on a simple subgraph of each graph exploration, which is justified because for $n$ large it behaves like a branching process (see \eqref{CM-new-2}).

%%%

\subsection{Convergence in the joint regime under a uniformity assumption}

In Theorem~\ref{thmA2} we considered non-backtracking exploration and did not address backtracking exploration. Under a uniformity assumption we can derive a limit for both explorations. Define
\begin{align*}
\Psi(N)=\sup_{k,n\geq N}\mathrm{d}_{P}(\mu_{n}^{(k)},\mu_{n}^{(\infty)}),
\end{align*}
where $\mathrm{d}_{P}$ denotes the Prohorov metric \cite[Section 6]{Billingsley}.

\begin{theorem} 
\label{thm4}
Suppose that the assumptions in Theorem \ref{thm3} hold. If $\Psi(N)\prob 0$ as $N\to\infty$, then $\mu_{n}^{(k_n)} \weak \mu$ as $n\to\infty$ in probability for any $k_n$ such that $\lim_{n\to\infty} k_n = \infty$.
\end{theorem}

%%%% SECTION 6 %%%%%%%%%%%%%%%%%%%%%%%%%%%%

\section{Proof of the main theorems}
\label{sec:proofs}

In this section we provide the proofs of the theorems in Sections~\ref{sec:RWRG}--\ref{sec:orderoflimits}. Section~\ref{ss:amlfb} proves Theorem~\ref{thm0}. Section~\ref{ss:cv} proves Theorems~\ref{thm1}, \ref{thm2}, \ref{thm3} and \ref{thm6}. Section~\ref{sec:doublelimit} proves Theorems~\ref{thmA1}, \ref{thmA2} and \ref{thm4}. 

%%%

\subsection{Average multi-level friendship bias}
\label{ss:amlfb}

\begin{proof}[Proof of Theorem \ref{thm0}]
First consider the non-backtracking exploration. For $k\in\N$, using \eqref{transitions1}--\eqref{transitions2} and symmetrisation we get
\begin{align*}
\dfrac{1}{n}\sum_{i_{0}\in [n]}\sum_{i_{k}\in [n]} P_{n}^{(k)}(i_{0},i_{k}) \,\dn_{i_{k}} 
&=\dfrac{1}{2n}\sum_{i_{0}\in [n]}\sum_{i_{k}\in [n]}\Big(P_{n}^{(k)}(i_{0},i_{k})\,
\dn_{i_{k}}+P_{n}^{(k)}(i_{k},i_{0})\,\dn_{i_{0}}\Big)\\
&=\dfrac{1}{2n}\sum_{(i_{0},i_{1},\ldots ,i_{k})\in W_{n,k}}\bigg(\frac{\dn_{i_{k}}}{\dn_{i_{0}}}
+\frac{\dn_{i_{0}}}{\dn_{i_{k}}}\bigg)\,\frac{1}{\prod_{l=1}^{k-1}(\dn_{i_{l}}-1)}
\end{align*}
and
\begin{align*}
\dfrac{1}{n}\sum_{i_{0}\in [n]} \dn_{i_{0}}
&=\dfrac{1}{2n}\sum_{i_{0}\in [n]}\sum_{i_{k}\in [n]}\Big(P_{n}^{(k)}(i_{0},i_{k})\,\dn_{i_{0}}
+P_{n}^{(k)}(i_{k},i_{0})\,\dn_{i_{k}}\Big)\\
&=\dfrac{1}{n}\sum_{(i_{0},i_{1},\ldots ,i_{k})\in W_{n,k}}\frac{1}{\prod_{l=1}^{k-1}(\dn_{i_{l}}-1)}.
\end{align*}
Hence
\begin{align*}
\Delta_{[n]}^{(k)} &=\dfrac{1}{2n}\sum_{(i_{0},i_{1},\ldots ,i_{k})\in W_{n,k}}\left(\sqrt{\frac{\dn_{i_{k}}}{\dn_{i_{0}}}}
-\sqrt{\frac{\dn_{i_{0}}}{\dn_{i_{k}}}}\,\right)^{2}\,\frac{1}{\prod_{l=1}^{k-1}(\dn_{i_{l}}-1)}\geq 0,
\end{align*}
with equality if and only if the degrees of the endpoints of each non-backtracking walk in $W_{n,k}$ are the same. 

Next consider the backtracking exploration. For $k\in\N$, set 
\[
V_{n,k}=\Big\{(i_{0},i_{1},\ldots ,i_{k})\in [n]^{k+1}\colon\,\{i_{j},i_{j+1}\}\in E(G_{n})\text{ for all }0\leq j\leq k-1\Big\}.
\]
Then \eqref{BT-transition}, along with the Chapman-Kolmogorov equation, implies that
\begin{align*}
\dfrac{1}{n}\sum_{i_{0}\in [n]}\sum_{i_{k}\in [n]} P_{n}^{(k)}(i_{0},i_{k}) \,\dn_{i_{k}} 
&=\dfrac{1}{2n}\sum_{i_{0}\in [n]}\sum_{i_{k}\in [n]}\Big(P_{n}^{(k)}(i_{0},i_{k})\,\dn_{i_{k}}
+P_{n}^{(k)}(i_{k},i_{0})\,\dn_{i_{0}}\Big)\\
&=\dfrac{1}{2n}\sum_{(i_{0},i_{1},\ldots ,i_{k})\in V_{n,k}}\bigg(\frac{\dn_{i_{k}}}{\dn_{i_{0}}}
+\frac{\dn_{i_{0}}}{\dn_{i_{k}}}\bigg)\,\frac{1}{\prod_{l=1}^{k-1}\dn_{i_{l}}}
\end{align*}
and
\begin{align*}
\dfrac{1}{n}\sum_{i_{0}\in [n]} \dn_{i_{0}}
&=\dfrac{1}{2n}\sum_{i_{0}\in [n]}\sum_{i_{k}\in [n]}\Big(P_{n}^{(k)}(i_{0},i_{k})\,\dn_{i_{0}}
+P_{n}^{(k)}(i_{k},i_{0})\,\dn_{i_{k}}\Big)\\
&=\dfrac{1}{n}\sum_{(i_{0},i_{1},\ldots ,i_{k})\in V_{n,k}}\frac{1}{\prod_{l=1}^{k-1}\dn_{i_{l}}}.
\end{align*}
Hence
\begin{align*}
\Delta_{[n]}^{(k)} &=\dfrac{1}{2n}\sum_{(i_{0},i_{1},\ldots ,i_{k})\in V_{n,k}}\left(\sqrt{\frac{\dn_{i_{k}}}{\dn_{i_{0}}}}
-\sqrt{\frac{\dn_{i_{0}}}{\dn_{i_{k}}}}\,\right)^{2}\,\frac{1}{\prod_{l=1}^{k-1}\dn_{i_{l}}} \geq 0,
\end{align*}
with equality if and only if the degrees of the endpoints of each backtracking walk in $V_{n,k}$ are the same. 

We continue with the backtracking exploration. For odd values of $k$, by traversing $k$ times from each edge, it is seen that the equality holds if and only if every connected component of the graph is regular. But for even values of $k$ this is equivalent to the case where each connected component of the graph is either regular or bi-regular bipartite. Indeed, in the case when each connected component of the graph is either regular or bi-regular bipartite, and $k$ is even, it is clear that the degrees of the endpoints of each backtracking walk in $V_{n,k}$ are equal. However, to prove the converse, first note that for any even $l$, equality of endpoint degrees in walks of $V_{n,l}$ holds if and only if the same equality holds for walks in $V_{n,2}$. Now, consider an arbitrary connected component of $G_{n}$, say $C_n$, and let $k$ be an even integer. Suppose that the degrees of the endpoints of each walk in $V_{n,k}$ are equal and that $C_{n}$ is not a bipartite graph. Since $C_{n}$ is not bipartite, there is an odd cycle in $C_{n}$. By traversing the odd cycle as many times as needed, any two vertices of $C_{n}$ can be joined by a walk in $V_{n,l}$ with some even $l$, which ensures that they must have the same degrees, indicating that $C_{n}$ is regular. On the other hand, if the degrees of the endpoints of each walk in $V_{n,k}$ are the same and $C_{n}$ is a bipartite graph, then any two vertices on the same side can be joined by a walk in $V_{n,l}$ with an even $l$. Therefore they must have the same degrees, implying that $C_{n}$ is bi-regular bipartite.
\end{proof}

%%%

\subsection{Convergence of the measures}
\label{ss:cv}
In this section we analyse the convergence of the associated measures. We refer to Section~\ref{sec:shortlevel} for the definition of local convergence.

%%%

\subsubsection{Convergence for short levels}

\begin{proof}[Proof of Theorem \ref{thm1}]
We give the proof for the non-backtracking exploration only. A similar approach can be applied to the backtracking exploration.

Define an everywhere locally finite modification of $(G_{\infty},\phi)$ by
\begin{align}
\label{Ginfinity}
(G_{\infty}^{\prime}(\omega),\phi) = 
\left\{
\begin{array}{ll}
(G_{\infty}(\omega),\phi), &\text{if }(G_{\infty}(\omega),\phi) \text{ is locally finite}, \\
(H,\phi), &\text{otherwise},
\end{array} \right.
\end{align}
where $(H,\phi)$ is an arbitrary locally finite infinite deterministic rooted tree (rooted at $\phi$). To simplify the notation, we assume, without loss of generality, that $(G_{\infty}, \phi) \equiv (G_{\infty}^{\prime}, \phi)$. This modification ensures that the quantities defined below are well defined for all realisations, since $(G_\infty^\prime,\phi)$ is now always locally finite. Moreover, let $\mathscr{C}(G_n, j)$ denote the connected component of vertex $j \in V(G_n)$ in the graph $G_n$, viewed as a rooted graph with $j$ as the root vertex. When $G_n$ is connected, we simply have $\mathscr{C}(G_{n},j)=(G_{n},j)$. 

Fix $k\in\N$. Let $U_{n}$ be a uniformly chosen vertex from $[n]$. The proof below here follows the lines of the proof of \cite[Theorems 2.3--2.4]{HHP}, and is a direct consequence of the definition of local convergence of random graphs and weak convergence of measures.

\medskip\noindent
(a) Since $G_{n}$ converges locally in probability to $(G_{\infty},\phi)$, for every bounded and continuous function $h\colon\,(\mathscr{G},\mathrm{d}_{\mathscr{G}}) \to (\R,\vert\cdot\vert)$ we have 
\begin{equation}
\label{a-1}
\E_{\bar{\Lambda}_{n}}\big[h(\mathscr{C}(G_{n},U_{n})) \mid G_{n}\big] 
= \dfrac{1}{n} \sum_{i \in [n]} h(\mathscr{C}(G_{n},i)) \prob \E_{\bar{\Lambda}}\big[h((G_{\infty},\phi))\big],
\qquad n\to\infty,
\end{equation}
where the expectation in the left-hand side is the conditional expected value given $G_{n}$ when $(G_{n},U_{n})$ has the joint law $\bar{\Lambda}_{n}$. 

Now, let $f\colon\,(\R ,\vert\cdot\vert) \to (\R ,\vert\cdot\vert)$ be an arbitrary bounded and continuous function. Define the function $h\colon\,(\mathscr{G},\mathrm{d}_{\mathscr{G}}) \to (\R,\vert\cdot\vert)$ as follows:
\begin{equation}
\label{h}
h((G,o)) = (f\circ g)((G,o)),
\end{equation}
where, with $i_{0}=o$, the function $g\colon\,(\mathscr{G},\mathrm{d}_{\mathscr{G}}) \to (\R,\vert\cdot\vert)$ is defined as $g((G,o))=\Delta_{o,k}^{(G)}$ with
\begin{equation}
\label{deltaoG}
\Delta_{o,k}^{(G)}=\sum_{(i_{1},\ldots,i_{k})\in (V(G))^{k}}\frac{A^{(G)}_{i_{0},i_{1}}}{d^{(G)}_{i_0}}
\prod_{l=1}^{k-1}\frac{A^{(G)}_{i_{l},i_{l+1}}\mathbbm{1}_{\{i_{l-1}\neq i_{l+1}\}}}{d^{(G)}_{i_{l}}-1}\,
d^{(G)}_{i_{k}}-d^{(G)}_{i_{0}}.
\end{equation}
Note that the boundedness of $f$ implies the boundedness of $h$. Also, for any $(G_1,o_1),(G_2,o_2)\in\mathscr{G}$, if $ \mathrm{d}_{\mathscr{G}}((G_1,o_1),(G_2,o_2))<\frac{1}{k+1}$, then 
\begin{equation*}
\sup\big\{r\geq 0\colon\,B_{r}^{(G_1)}(o_1) \simeq B_{r}^{(G_2)}(o_2)\big\}>k.
\end{equation*}
This implies that $B_{r}^{(G_1)}(o_1) \simeq B_{r}^{(G_2)}(o_2)$ for all $r\in\{0,1,\ldots ,k\}$. Consequently, $g((G_1,o_1))=g((G_2,o_2))$, which demonstrates that $g$ is continuous. Therefore $h$, being a composition of two continuous functions, is also continuous. Furthermore, $h(\mathscr{C}(G_{n},i)) = f(\Delta_{i,n}^{(k)})$ and $h((G_{\infty},\phi))=f(\Delta_{\phi}^{(k)})$ a.s. Inserting this into \eqref{a-1}, we get
\begin{equation*}
\int_\R f\dd\mu_{n}^{(k)} =\dfrac{1}{n} \sum_{i \in [n]} f(\Delta_{i,n}^{(k)}) \prob 
\E_{\bar{\Lambda}}\big[f(\Delta_{\phi}^{(k)})\big]=\int_\R f\dd\mu^{(k)}_{\infty},
\end{equation*}
which completes the proof that $\mu_{n}^{(k)} \weak \mu^{(k)}_{\infty}$ as $n\to\infty$ in probability. 

Suppose that $A\in\mathcal{B}(\R)$, and consider the function $h\colon\,(\mathscr{G},\mathrm{d}_{\mathscr{G}}) \to (\R,\vert\cdot\vert)$ defined by 
\begin{align}
\label{hindicator}
h((G,o))=\mathbbm{1}_{\{\Delta_{o,k}^{(G)}\in A\}},
\end{align}
with $\Delta_{o,k}^{(G)}$ defined as in \eqref{deltaoG}. For any $(G_1,o_1),(G_2,o_2)\in\mathscr{G}$, if $ \mathrm{d}_{\mathscr{G}}((G_1,o_1),(G_2,o_2))<\frac{1}{k+1}$, then $B_{r}^{(G_1)}(o_1) \simeq B_{r}^{(G_2)}(o_2)$ for $r\in\{0,1,\ldots ,k\}$. Hence, $h((G_1,o_1))=h((G_2,o_2))$, which implies that $h$ is continuous. In this way, from \eqref{a-1}, we get $\mu_{n}^{(k)}(A)\prob\mu^{(k)}_{\infty}(A)$. Applying the dominated convergence theorem, we also get that $\tilde{\mu}_{n}^{(k)}(A)\to\mu^{(k)}_{\infty}(A)$.

\medskip\noindent
(b) Let $\mu_{i,n}^{(k)}(\cdot):=\E_{n}[\delta_{\Delta_{i,n}^{(k)}}(\cdot)]$, which is the law of $\Delta_{i,n}^{(k)}$. For an arbitrary bounded and continuous function $f\colon\,(\R,\vert\cdot\vert) \to (\R,\vert\cdot\vert)$ we have
\begin{align}
\label{weaknew0}
\int_\R f \dd\tilde{\mu}_{n}^{(k)} = \dfrac{1}{n} \sum_{i \in [n]} \int_\R f \dd\mu_{i,n}^{(k)}
=\dfrac{1}{n}\sum_{i \in [n]} \E_{n}\big[f(\Delta_{i,n}^{(k)})\big].
\end{align}
Let $h$ be as in \eqref{h}. Since $G_{n}$ converges locally weakly to $(G_{\infty},\phi)$, we must have 
\begin{align}
\label{weaknew1}
\E_{\bar{\Lambda}_{n}}\big[h(\mathscr{C}(G_{n},U_{n}))\big] \to \E_{\bar{\Lambda}}\big[h((G_{\infty},\phi))\big],
\qquad n\to\infty,
\end{align}
where the expectation in the left-hand side is the expected value when $(G_{n},U_{n})$ has the joint law $\bar{\Lambda}_{n}$. But 
\begin{align}
\E_{\bar{\Lambda}_{n}}\big[h(\mathscr{C}(G_{n},U_{n}))\big] 
= \E_{n}\Big[\E_{\bar{\Lambda}_{n}}\big[h(\mathscr{C}(G_{n},U_{n}))\mid G_{n}\big]\Big]
&= \E_{n}\Big[\dfrac{1}{n}\sum_{i \in [n]}h(\mathscr{C}(G_{n},i))\Big] \label{weaknew02}\\
&=\dfrac{1}{n}\sum_{i \in [n]} \E_{n}\big[f(\Delta_{i,n}^{(k)})\big] \label{weaknew2}
\end{align}
and $\E_{\bar{\Lambda}}[h((G_{\infty},\phi))]=\E_{\bar{\Lambda}}[f(\Delta_{\phi}^{(k)})]$. Hence, by \eqref{weaknew0}, \eqref{weaknew1} and \eqref{weaknew2},
\begin{align*}
\int_\R f \dd\tilde{\mu}_{n}^{(k)} 
\to \E_{\bar{\Lambda}}\big[f(\Delta_{\phi}^{(k)})\big] = \int_\R f\dd \mu^{(k)}_{\infty},
\end{align*}
which completes the proof that $\tilde{\mu}_{n}^{(k)} \weak \mu^{(k)}_{\infty}$ as $n\to\infty$.

Finally, let $A\in\mathcal{B}(\R)$. By substituting the bounded and continuous function $h$ defined in \eqref{hindicator} into \eqref{weaknew1}, and using \eqref{weaknew02}, we conclude that $\tilde{\mu}_{n}^{(k)}(A)\to\mu^{(k)}_{\infty}(A)$.
\end{proof}

%%%

\subsubsection{Identification of the limit for short levels}
In this section, we prove Theorem \ref{thm2}. To begin, we shall first state a general lemma that will also be utilised in the proof of Theorem \ref{thm2}.
\begin{lemma}\label{lem0}
Let $(H_{n})_{n\in\N}$ be a sequence of finite non-null random graphs that converges locally weakly to an almost surely locally finite, connected rooted graph $(H, \varphi)$ with a deterministic law $\nu$. If the degrees of the neighbours of the root $\varphi$ are i.i.d., and independent of $d_{\varphi}^{(H)}$, and if $d_{\varphi}^{(H)}$ has a finite first moment, then the degree distribution of the neighbours of $\varphi$ is size-biased with respect to the degree distribution of $\varphi$.
\end{lemma}
\begin{proof}
Without loss of generality, assume that $(H, \varphi)$ is everywhere locally finite, following a similar approach to the construction in \eqref{Ginfinity}. As an extension of $(\mathscr{G}, \mathrm{d}_{\mathscr{G}})$, we define $\mathscr{G}_{\star}$ as the topological space of isomorphism classes of connected locally finite graphs with an ordered pair of root vertices, equipped with the topology induced by a natural extension of the metric $\mathrm{d}_{\mathscr{G}}$. 

Since $(H, \varphi)$ is a local weak limit, the probability measure $\nu$ is unimodular and obeys
the mass-transport principle \cite[Section 3.2]{BS}. Therefore, it is involution invariant \cite[Propositon 2.2]{AL}, i.e. for all Borel function $h:\mathscr{G}_{\star}\to [0,\infty)$,
\begin{align*}
\E_{\nu}\bigg[\sum_{w\in \partial_{\varphi}} h\big((H, \varphi ,w)\big)\bigg]=\E_{\nu}\bigg[\sum_{w\in\partial_{\varphi}} h\big((H, w,\varphi)\big)\bigg],
\end{align*}
where $\E_{\nu}$ is the expectation with respect to the random graph $(H, \varphi)$, and $\partial_{\varphi}=V(B_{1}^{(H)}(\varphi))\setminus\{\varphi\}$ is the set of neighbours of $\varphi$. In particular, for each $k\in\N_{0}$, by considering the Borel function $h_{k}$ on $\mathscr{G}_{\star}$ as $h_{k}((G,u,v))=\mathbbm{1}_{\{d_{v}^{(G)}=k\}}$, we have
\begin{align*}
\E_{\nu}\bigg[\sum_{w\in \partial_{\varphi}} \mathbbm{1}_{\{d_{w}^{(H)}=k\}}\bigg]=\E_{\nu}\bigg[k\,\mathbbm{1}_{\{d_{\varphi}^{(H)}=k\}}\bigg]=k\,\P_{\nu}\big\{d_{\varphi}^{(H)}=k\big\},
\end{align*}
where $\P_{\nu}$ is the probability measure on the underlying space of the random graph $(H, \varphi)$. Noting that the degrees of the neighbours of the root $\varphi$ are i.i.d. and independent of $d_{\varphi}^{(H)}$, we conclude that if $w_{0}$ is a neighbour of $\varphi$, then
\begin{align*}
\E_{\nu}\bigg[\sum_{w\in \partial_{\varphi}} \mathbbm{1}_{\{d_{w}^{(H)}=k\}}\bigg]
&=\sum_{j=0}^{\infty}\E_{\nu}\Big[\sum_{w\in \partial_{\varphi}} \mathbbm{1}_{\{d_{w}^{(H)}=k\}}\mid d_{\varphi}^{(H)}=j \Big]\,\P_{\nu}\big\{d_{\varphi}^{(H)}=j\big\}
\\&=\E_{\nu}\big[d_{\varphi}^{(H)}\big]\,\P_{\nu}\big\{d_{w_{0}}^{(H)}=k\big\}.
\end{align*}
Therefore
\begin{align*}
\E_{\nu}\big[d_{\varphi}^{(H)}\big]\,\P_{\nu}\big\{d_{w_{0}}^{(H)}=k\big\}=k\,\P_{\nu}\big\{d_{\varphi}^{(H)}=k\big\}.
\end{align*}
In other words, by noting that $0<\E_{\nu}[d_{\varphi}^{(H)}]<\infty$, we have
\begin{align*}
\P_{\nu}\big\{d_{w_{0}}^{(H)}=k\big\}=\dfrac{k\,\P_{\nu}\big\{d_{\varphi}^{(H)}=k\big\}}{\E_{\nu}[d_{\varphi}^{(H)}]},
\end{align*}
which completes the proof of the lemma.
\end{proof}

\begin{proof}[Proof of Theorem \ref{thm2}]
(a) Let us label the vertices of $(G_{\infty}, \phi)$ excluding the root by $\mathbb{N}$ in an arbitrary manner, i.e., $V(G_\infty) = \mathbb{N} \cup \{\phi\}$. Let $m_{1}=\E_{\bar{\Lambda}}[d_{1}]$ and $m_{2}=\E_{\bar{\Lambda}}[d_{1}^{2}]$. For $k\in\N$, set 
\[
\chi_{k}=\sum_{j\in \N} P^{k}(\phi, j)\, d^{(G_\infty)}_{j}.
\]
To complete the proof, it suffices to show that $\chi_{k}\prob m_1+1$ as $k\to\infty$ for the non-backtracking exploration, where the additional $1$ comes from the fact that $m_1$ represents only the expected number of descendants. This immediately allows us to conclude, by applying the continuous mapping theorem, that
\[ 
\Delta_{\phi}^{(k)}=\chi_{k}-d_{\phi}\prob m_1 +1-d_{\phi},\quad\quad k\to\infty,
\]
which settles the claim by noting that, according to Lemma \ref{lem0}, $m_{1}+1=m^{(2)}/m^{(1)}$.

For the non-backtracking exploration, 
\[
\chi_{k}=\sum_{j\in \N} P^{k}(\phi, j)\, d_{j}+1.
\]
 Using conditions (I) and (II), we have that $d_j$ and $P^k(\phi, j)$ are independent for each $j \in \N$. Moreover, $P^k(\phi, j)$ is non-zero only if $j$ is in the $k$th-generation after $\phi$ on the tree, i.e. when $\mathrm{dist}_{G_{\infty}}(\phi, j)=k$. Therefore, $P^k(\phi, i)P^k(\phi, j)$ is non-zero only if both $i$ and $j$ are in the $k$th-generation, and in this case, if $i \neq j$, we must have that $P^k(\phi, i)P^k(\phi, j)$ is independent of $d_i$ and $d_j$.
Therefore, by the monotone convergence theorem,
\begin{align}\label{nk0}
\E_{\bar{\Lambda}}\big[\chi_{k}\big]=m_1+1
\end{align}
and
\begin{align*}
\E_{\bar{\Lambda}}\bigg[\sum_{i\in\N}\sum_{\substack{j \in \N \\ j \neq i}}P^{k}(\phi, i)\,P^{k}(\phi, j)\,d_{i}\, d_{j}\bigg]\leq m_{1}^{2}.
\end{align*}
Hence, by the monotone convergence theorem,
\begin{align*}
\E_{\bar{\Lambda}}\big[\chi_{k}^{2}\big]&\leq (m_{1}+1)^{2}+m_{2}\,\E_{\bar{\Lambda}}\bigg[\sum_{i\in\N}\big(P^{k}(\phi, i)\big)^2\bigg]\\&
\leq (m_{1}+1)^{2}+m_{2}\,\E_{\bar{\Lambda}}\bigg[\sup_{i\in\N}P^{k}(\phi, i)\bigg].
\end{align*}
In the following, we prove that 
\begin{align}\label{Esup}
\E_{\bar{\Lambda}}\bigg[\sup_{i\in\N}P^{k}(\phi, i)\bigg]
\to 0\quad \text{as}\quad k\to\infty.
\end{align}
Consequently, by applying Chebyshev's inequality, we can conclude that $\chi_{k}\prob m_1+1$ as $k\to\infty$. 

To prove \eqref{Esup}, note that \eqref{a-1} holds for every bounded continuous function $h\colon (\mathscr{G}, \mathrm{d}_{\mathscr{G}}) \to (\mathbb{R}, |\cdot|)$. If $h((G,o))$ is the indicator function of the event that not all degrees in the graph $G$ are at least 3 within some arbitrary finite graph distance from $o$, then (similarly as in the proof of the continuity of \eqref{hindicator} in the proof of Theorem \ref{thm1})  it follows that $h$ is continuous. 
Indeed, the value of $h((G,o))$ depends only on the finite rooted ball $B_r^{(G)}(o)$ for some fixed radius $r$. 
Since the topology induced by $\mathrm{d}_{\mathscr{G}}$ corresponds to local convergence of rooted graphs, any function that depends only on a finite neighbourhood of the root is continuous with respect to this topology. For sufficiently large $n$, since the minimum degree in the graph $G_n$ is at least 3, applying this function $h$ we get that the left-hand side of \eqref{a-1} is zero, which implies that the right-hand side of \eqref{a-1} must also be zero. As the distance is arbitrary, it follows that $d_\phi \geq 3$ and $\inf_{i \in \mathbb{N}} d_i \geq 2$ almost surely. Therefore $\sup_{i \in \mathbb{N}} P^k(\phi, i) < 2^{-k}$ almost surely, which proves \eqref{Esup}.

\medskip\noindent
(b) 
Next consider the backtracking exploration. Let $\Omega^{\prime}\in \Omega_{\bar{\Lambda}}$ be such that for all $\omega\in \Omega^{\prime}$ the tree $(G_{\infty}(\omega),\phi)$ is finite and non-bipartite. Then the backtracking random walk $X(\omega)$ on it is ergodic, and so
\begin{align*}
\lim_{k\rightarrow\infty}P^{k}(\phi,j)(\omega) = 
\frac{d_{j}^{(G_{\infty})}(\omega)}{\sum_{j\in V(G_{\infty}(\omega))}d_{j}^{(G_{\infty})}(\omega)}
\end{align*}
for all $j\in V(G_{\infty}(\omega))$ \cite[Theorem 1.8.3]{N}. Hence, as $k\to\infty$,
\begin{align*}
\Delta_{\phi}^{(k)}(\omega)=\sum_{j\in V(G_{\infty}(\omega))} P^{k}(\phi, j)(\omega)\, 
d^{(G_\infty)}_{j}(\omega)-d_{\phi}(\omega)\to \frac{\sum_{j\in V(G_{\infty}(\omega))}
(d_{j}^{(G_{\infty})}(\omega))^2}{\sum_{j\in V(G_{\infty}(\omega))}d_{j}^{(G_{\infty})}(\omega)}
-d_{\phi}(\omega).
\end{align*}
Since $\P_{\bar{\Lambda}}\{\Omega^{\prime}\}=1$, this completes the proof.
\end{proof}

%%%

\subsubsection{Convergence and identification of the limit for long levels}

In this section, we prove Theorem \ref{thm3}. To proceed, we first state a general lemma that will also be used in the proof.

\begin{lemma}
\label{lem}
If $(H_{n})_{n\in\N}$ is a sequence of finite random graphs with a deterministic vertex set that converges locally in probability to an almost surely locally finite, connected rooted graph $(H, \varphi)$, and if $((d_{i}^{(H_{n})})^{m})_{i\in V(H_{n}),\, n\in\N}$ is uniformly integrable for some $m \in \N$, then
\begin{align*}
\frac{1}{\#V(H_{n})}\sum_{i\in V(H_{n})}(d_{i}^{(H_{n})})^{j}\prob \E_{H}\big[(d_{\varphi}^{(H)})^{j}\big],\qquad n\to\infty ,
\end{align*}
for all $j\in\{1,\ldots,m\}$, where $\E_{H}$ denotes the expectation with respect to the randomness of the graph $(H, \varphi)$. 
\end{lemma}
\begin{proof}
Consider $j \in \{1, \ldots, m\}$. Without loss of generality, assume that $(H, \varphi)$ is everywhere locally finite, in a manner similar to the construction in \eqref{Ginfinity}. 

For $M \in \N$, define a bounded function $h_{j,M}\colon\,(\mathscr{G},\mathrm{d}_{\mathscr{G}}) \to (\R,\vert\cdot\vert)$ by setting
\begin{align*}
h_{j,M}((G,o))=(d^{(G)}_{o})^{j}\,\mathbbm{1}_{\{d^{(G)}_{o}\leq M\}}.
\end{align*}
Then, in a manner analogous to the proof of the continuity of the function in \eqref{hindicator} provided in the proof of Theorem \ref{thm1}, it can be easily shown that $h_{j,M}$ is continuous. Therefore, by the local convergence in probability of $(H_{n})_{n\in\N}$ to $(H, \varphi)$, we have
\begin{align}\label{lem1}
\frac{1}{\#V(H_{n})}\sum_{i\in V(H_{n})}(d_{i}^{(H_{n})})^{j}\,\mathbbm{1}_{\{d_{i}^{(H_{n})}\leq M\}}\prob \E_{H}\Big[(d_{\varphi}^{(H)})^{j}\,\mathbbm{1}_{\{d_{\varphi}^{(H)}\leq M\}}\Big],\qquad n\to\infty .
\end{align}
Moreover, by the monotone convergence theorem, we have that
\begin{align*}
\E_{H}\Big[(d_{\varphi}^{(H)})^{j}\,\mathbbm{1}_{\{d_{\varphi}^{(H)}\leq M\}}\Big]\to\E_{H}\big[(d_{\varphi}^{(H)})^{j}\big]\qquad M\to\infty .
\end{align*}
Therefore, for any $\epsilon >0$, we can take $M_{\epsilon}\in\N$ such that for all $M\geq M_{\epsilon}$,
\begin{align*}
\bigg|\E_{H}\Big[(d_{\varphi}^{(H)})^{j}\,\mathbbm{1}_{\{d_{\varphi}^{(H)}\leq M\}}\Big]-\E_{H}\big[(d_{\varphi}^{(H)})^{j}\big]\bigg|<\frac{\epsilon}{4},
\end{align*}
and by the Markov inequality,
\begin{align*}
&\P_{H_{n}}\bigg\{\Big |\tfrac{1}{\#V(H_{n})}\sum_{i\in V(H_{n})}(d_{i}^{(H_{n})})^{j}-\E_{H}\big[(d_{\varphi}^{(H)})^{j}\big]\Big |>\epsilon\bigg\}
\\&\leq \P_{H_{n}}\bigg\{\Big |\tfrac{1}{\#V(H_{n})}\sum_{i\in V(H_{n})}(d_{i}^{(H_{n})})^{j}\,\mathbbm{1}_{\{d_{i}^{(H_{n})}\leq M\}}-\E_{H}\big[(d_{\varphi}^{(H)})^{j}\big]\Big|>\tfrac{\epsilon}{2}\bigg\}\\&\quad
+\P_{H_{n}}\bigg\{\tfrac{1}{\#V(H_{n})}\sum_{i\in V(H_{n})}(d_{i}^{(H_{n})})^{j}\,\mathbbm{1}_{\{d_{i}^{(H_{n})}> M\}}>\tfrac{\epsilon}{2}\bigg\}\\&
\leq \P_{H_{n}}\bigg\{\Big |\tfrac{1}{\#V(H_{n})}\sum_{i\in V(H_{n})}(d_{i}^{(H_{n})})^{j}\,\mathbbm{1}_{\{d_{i}^{(H_{n})}\leq M\}}-\E_{H}\big[(d_{\varphi}^{(H)})^{j}\,\mathbbm{1}_{\{d_{\varphi}^{(H)}\leq M\}}\big]\Big|>\tfrac{\epsilon}{4}\bigg\}\\&\quad
+\frac{2}{\epsilon}\,\sup_{n\in\N}\,\sup_{i\in V(H_{n})}\E_{H_{n}}\bigg[(d_{i}^{(H_{n})})^{m}\,\mathbbm{1}_{\{d_{i}^{(H_{n})}> M\}}\bigg],
\end{align*}
where $\P_{H_{n}}$ denotes the probability measure on the underlying space of the random graph $H_{n}$, and $\E_{H_{n}}$ represents the expectation with respect to the randomness of $H_{n}$. By taking the limit as $n\to\infty$ and then $M\to\infty$, and applying \eqref{lem1} along with the uniform integrability condition, we establish the lemma.
\end{proof}

\begin{proof}[Proof of Theorem \ref{thm3}]
Let $U_{n}$ be a uniformly chosen vertex from $[n]$. Also, let $\E_{\bar{\Lambda}_{n}}$ be defined as the proof of Theorem \ref{thm1}.

\medskip\noindent
(a) Fix $n\in\N_{3}$ for non-backtracking exploration and $n\in\N_{2}$ for backtracking exploration. From \eqref{statlim}, we have
\begin{align*}
\Delta_{U_n,n}^{(k)}-\Dst_{U_n,n} =\sum_{j \in [n]} \big[P_{n}^{(k)}(U_n,j)-\pi_{n}(j)\big] \,d_{j}^{(G_{n})} 
\as 0, \qquad k\to \infty .
\end{align*}
Now, let $f\colon\,(\R ,|\cdot|) \to (\R ,|\cdot|)$ be an arbitrary bounded and continuous function. Then $f(\Delta_{U_n,n}^{(k)})-f(\Dst_{U_n,n} )\as 0$ as $k\to \infty$. In particular, by the dominated convergence theorem,
\begin{align*}
\int_\R f\dd\mu_{n}^{(k)} = \E_{\bar{\Lambda}_{n}}\big[f(\Delta_{U_n,n}^{(k)}) \mid G_{n}\big] \as \E_{\bar{\Lambda}_{n}}
\big[f(\Dst_{U_n,n} ) \mid G_{n}\big] = \int_\R f\dd\mu^{(\infty)}_{n}, \qquad k\to \infty,
\end{align*}
which establishes the claim.

\medskip\noindent
(b) For every bounded and continuous function $h\colon\,(\mathscr{G},\mathrm{d}_{\mathscr{G}}) \to (\R,\vert\cdot\vert)$ we have \eqref{a-1}. Suppose that
\begin{align}
\label{t62}
&m_{n}^{(1)} = \E_{\bar{\Lambda}_{n}}\Big[\dn_{U_{n}} \mid G_{n}\Big], 
\qquad
m_{n}^{(2)} = \E_{\bar{\Lambda}_{n}}\Big[\big(\dn_{U_{n}}\big)^{2} \mid G_{n}\Big]. 
\end{align}
Then, for $i\in[n]$, 
\begin{align*}
\Dst_{i,n}=\dfrac{m_{n}^{(2)}}{m_{n}^{(1)}}-\dn_{i}.
\end{align*}
Let $f\colon\,(\R ,\vert\cdot\vert) \to (\R ,\vert\cdot\vert)$ be an arbitrary bounded and uniform continuous function. To complete the proof by the Portmanteau theorem, it suffices to show that
\begin{align}
\label{kn1}
\E_{\bar{\Lambda}_{n}}\Big[f\Big(\frac{m_{n}^{(2)}}{m_{n}^{(1)}}-\dn_{U_n}\Big) \mid G_{n}\Big]\prob 
\E_{\bar{\Lambda}}\Big[f\Big(\frac{m^{(2)}}{m^{(1)}}-d_{\phi}\Big)\Big]=\E_{\bar{\Lambda}}\big[f(\Delta_{\phi})\big],
\qquad n\to\infty.
\end{align}

First, we note that the continuous mapping theorem, along with Lemma \ref{lem}, implies that
\begin{align}
\label{kn2}
\frac{m_{n}^{(2)}}{m_{n}^{(1)}} \prob \frac{m^{(2)}}{m^{(1)}},\qquad n\to\infty.
\end{align}
Now fix an arbitrary $\epsilon>0$. Define the event 
\begin{align*}
A_{n,\epsilon}=\bigg\{\Big |\frac{m_{n}^{(2)}}{m_{n}^{(1)}}-\frac{m^{(2)}}{m^{(1)}} \Big|\leq \epsilon\bigg\}.
\end{align*}
Note that $m^{(1)}$ and $m^{(2)}$ are deterministic because $\bar{\Lambda}$ is a deterministic measure. We have
\begin{equation}
\label{kn3}
\begin{aligned}
\E_{\bar{\Lambda}_{n}}\Big[f\Big(\frac{m_{n}^{(2)}}{m_{n}^{(1)}}-\dn_{U_n}\Big) ~\Big|~ G_{n}\Big]
&=\E_{\bar{\Lambda}_{n}}\Big[f\Big(\frac{m_{n}^{(2)}}{m_{n}^{(1)}}-\dn_{U_n}\Big) ~\Big|~ G_{n}\Big]
\mathbbm{1}_{A_{n,\epsilon}}\\
&\quad+\E_{\bar{\Lambda}_{n}}\Big[f\Big(\frac{m_{n}^{(2)}}{m_{n}^{(1)}}-\dn_{U_n}\Big) 
~\Big|~ G_{n}\Big]\mathbbm{1}_{A_{n,\epsilon}^{c}}.
\end{aligned}
\end{equation}
By \eqref{kn2} and the boundedness of $f$, the term in the second line of \eqref{kn3} converges to $0$ in probability. Defining the interval $K_{\epsilon} = [m^{(2)}/m^{(1)}-\epsilon,m^{(2)}/m^{(1)}+\epsilon]$ and two functions
\begin{align*}
&g_{1,\epsilon}(x) = \inf_{y\in K_{\epsilon}} f(y-x),\\
&g_{2,\epsilon}(x) = \sup_{y\in K_{\epsilon}} f(y-x),
\end{align*}
we also have
\begin{align}
\label{kn4}
&\E_{\bar{\Lambda}_{n}}\Big[g_{1,\epsilon}\big(\dn_{U_n}\big) ~\Big|~ G_{n}\Big]\mathbbm{1}_{A_{n,\epsilon}} \nonumber\\ 
&\leq \E_{\bar{\Lambda}_{n}}\Big[f\Big(\frac{m_{n}^{(2)}}{m_{n}^{(1)}}-\dn_{U_n}\Big) ~\Big|~ G_{n}\Big] 
\mathbbm{1}_{A_{n,\epsilon}}
\leq \E_{\bar{\Lambda}_{n}}\Big[g_{2,\epsilon}\big(\dn_{U_n}\big) ~\Big|~ G_{n}\Big]\mathbbm{1}_{A_{n,\epsilon}},
\end{align}
almost surely. Both $g_{1,\epsilon}$ and $g_{2,\epsilon}$ are continuous. Indeed,
if $x_n \to x$ as $n\to\infty$ in $(\R ,\vert\cdot\vert)$, then the uniform continuity of $f$ implies that
\[
\sup_{y \in K_{\epsilon}} \big|f(y- x_n) - f(y - x) \big|
\leq \omega_f(|x_n - x|) \to 0,\qquad n\to\infty,
\]
where $\omega_f$ denotes the modulus of continuity of $f$. For any real-valued functions $a,b$ defined on the set $K_\epsilon$, 
\[
\Big| \inf_{y\in K_\epsilon} a(y) - \inf_{y\in K_\epsilon} b(y) \Big| 
\leq \sup_{y\in K_\epsilon} |a(y)-b(y)|,
\qquad
\Big| \sup_{y\in K_\epsilon} a(y) - \sup_{y\in K_\epsilon} b(y) \Big|
\leq \sup_{y\in K_\epsilon} |a(y)-b(y)|.
\]
Applying this with $a(y) = f(y - x_n)$ and $b(y) = f(y - x)$ gives
\[
\big| g_{1,\epsilon}(x_n) - g_{1,\epsilon}(x) \big|
\leq \omega_f(|x_n - x|) \to 0,\qquad n\to\infty,
\]
and 
\[
\big| g_{2,\epsilon}(x_n) - g_{2,\epsilon}(x) \big|
\leq \omega_f(|x_n - x|) \to 0,\qquad n\to\infty.
\]
Hence $g_{1,\epsilon}$ and $g_{2,\epsilon}$ are continuous 
(indeed, uniformly continuous with the same modulus).
If $h_1((G,o)) = d_o^{(G)}$ for a rooted graph $(G,o)\in\mathscr{G}$, then, similarly to the proof of the continuity of \eqref{hindicator} in Theorem~\ref{thm1}, it follows that $h_1$ is continuous, since it depends only on the finite rooted ball $B_1^{(G)}(o)$. Now employing two real bounded continuous functions $g_{1,\epsilon}\circ h_{1}$ and $g_{2,\epsilon}\circ h_{1}$ to \eqref{a-1}, from the continuous mapping theorem, we obtain that
\begin{equation}
\label{kn7}
\begin{aligned}
&Y_{n,\epsilon} = \E_{\bar{\Lambda}_{n}}\Big[g_{1,\epsilon}\big(\dn_{U_n}\big) 
~\Big|~ G_{n}\Big]\mathbbm{1}_{A_{n,\epsilon}}\prob\E_{\bar{\Lambda}}\big[g_{1,\epsilon}\big(d_{\phi}\big)\big]
= Y_{\epsilon},\qquad n\to\infty,\\
&Z_{n,\epsilon} = \E_{\bar{\Lambda}_{n}}\Big[g_{2,\epsilon}\big(\dn_{U_n}\big) 
~\Big|~ G_{n}\Big]\mathbbm{1}_{A_{n,\epsilon}}\prob\E_{\bar{\Lambda}}\big[g_{2,\epsilon}\big(d_{\phi}\big)\big]
= Z_{\epsilon},\qquad n\to\infty.
\end{aligned}
\end{equation}

Let
\begin{align*}
X_{n,\epsilon} = \E_{\bar{\Lambda}_{n}}\left[f\left(\frac{m_{n}^{(2)}}{m_{n}^{(1)}}-\dn_{U_n}\right) ~\Bigg|~ G_{n}\right]
\mathbbm{1}_{A_{n,\epsilon}}.
\end{align*}
Consider an arbitrary $\eta>0$. By the dominated convergence theorem, both $Y_{\epsilon}$ and $Z_{\epsilon}$ converge to $\E_{\bar{\Lambda}}[f(\Delta_{\phi})]$ as $\epsilon\downarrow 0$. Hence, for small enough $\epsilon$, 
\[
\E_{\bar{\Lambda}}[f(\Delta_{\phi})]-\tfrac{\eta}{2}\leq Y_{\epsilon}\leq \E_{\bar{\Lambda}}[f(\Delta_{\phi})]
+ \tfrac{\eta}{2},
\] 
and 
\[
\E_{\bar{\Lambda}}[f(\Delta_{\phi})]-\tfrac{\eta}{2}\leq Z_{\epsilon}\leq\E_{\bar{\Lambda}}[f(\Delta_{\phi})]
+ \tfrac{\eta}{2}.
\] 
Therefore, using \eqref{kn4}, we can conclude that
\begin{align*}
\P_{n}\Big\{\E_{\bar{\Lambda}}[f(\Delta_{\phi})]-\eta\leq X_{n,\epsilon}
\leq \E_{\bar{\Lambda}}[f(\Delta_{\phi})] +\eta\Big\}
&\geq \P_{n}\Big\{Y_{\epsilon} - \tfrac{\eta}{2}\leq X_{n,\epsilon}\leq Z_{\epsilon} + \tfrac{\eta}{2}\Big\}\\
&\geq \P_{n}\Big\{\big| Y_{n,\epsilon}-Y_{\epsilon}\big |\leq \tfrac{\eta}{2},\,\big| Z_{n,\epsilon}-Z_{\epsilon}\big|
\leq \tfrac{\eta}{2}\Big\}.
\end{align*}
Letting $n\to\infty$, we obtain from \eqref{kn7} that $X_{n,\epsilon}\prob \E_{\bar{\Lambda}}[f(\Delta_{\phi})]$ as $n\to\infty$. This, together with \eqref{kn3} and the continuous mapping theorem, implies \eqref{kn1}, which settles the claim. 
\end{proof}

\begin{proof}[Proof of Theorem \ref{thm6}]
Let $U_{n}$ be a uniformly chosen vertex from $[n]$. Also, let $\E_{\bar{\Lambda}_{n}}$ be defined as the proof of Theorem \ref{thm1}.

\medskip\noindent
(a) Fix $n\in\N_{2}$. From \eqref{statlim2}, we have
\begin{align*}
\Delta_{U_n,n}^{(k,\mathrm{lazy})}-\Delta^{(\mathrm{st},\mathrm{c})}_{U_n,n} =\sum_{j \in [n]} \big[P_{n}^{(k,\mathrm{lazy})}(U_n,j)-\pi_{n}^{(\mathrm{c})}(j)\big] \,d_{j}^{(G_{n})} 
\to 0, \qquad k\to \infty ,
\end{align*}
everywhere. Thus, the result follows by the same approach as in the proof of Theorem \ref{thm3}(a).

\medskip\noindent
(b) For each $i\in[n]$, we have
\begin{align*}
\Delta^{(\mathrm{st},\mathrm{c})}_{i,n} = \sum_{j\in [n]}\frac{(\dn_{j})^2}{\sum_{l\in V(C_{j})}\dn_{l}}-\dn_{i}.
\end{align*}
To complete the proof, it suffices to show that
\begin{align}
\label{t61}
M_{n}=\sum_{j\in [n]}\frac{(\dn_{j})^2}{\sum_{l\in V(C_{j})}\dn_{l}} \prob \frac{m^{(2)}}{m^{(1)}},\qquad n\to\infty.
\end{align}
This allows us to apply a similar argument, as used in the proof of Theorem \ref{thm3}(b), to establish the theorem. 

Let $m_{n}^{(1)}$ and $m_{n}^{(2)}$ be defined as in \eqref{t62}. For any $\epsilon>0$,
\begin{align*}
\P_{n}\bigg\{\Big|M_{n}-\frac{m^{(2)}}{m^{(1)}}\Big|>\epsilon\bigg\}&
\leq \P_{n}\bigg\{\Big|\frac{m_{n}^{(2)}}{m_{n}^{(1)}}-\frac{m^{(2)}}{m^{(1)}}\Big|>\epsilon\bigg\}+\P_{n}\Big\{G_{n}\text{ is disconnected}\Big\}\to 0,\quad n\to\infty,
\end{align*}
by \eqref{kn2} and the assumption of the theorem. This completes the proof.
\end{proof}

%%%

\subsection{Convergence in the double limit}
\label{sec:doublelimit}

\begin{proof}[Proof of Theorem \ref{thmA1}]
Let $U_{n}$ be a uniformly chosen vertex from $[n]$. Also, let $\E_{\bar{\Lambda}_{n}}$ be defined as in the proof of Theorem \ref{thm1}. Set $M=\sup_{n}\sup_{i\in[n]}d_{i}^{(G_{n})}$. For $n\in\N_{3}$ in the non-backtracking exploration and $n\in\N_{2}$ in the backtracking exploration, we have
\begin{align*}
\Delta_{U_n,n}^{(k_n)}-\Dst_{U_n,n} =\sum_{j \in [n]} \big[P_{n}^{(k_n)}(U_n,j)-\pi_{n}(j)\big] \,
d_{j}^{(G_{n})} \leq 2 M \mathcal{D}_{n}(k_{n}).
\end{align*}
Here, we recall that $\Dst_{i,n}$ is defined in \eqref{Deltast}. Since $\lim_{n\to\infty} k_{n}/\psi_{n}=\infty$, we have that 
\begin{align}
\label{thmA11}
\Delta_{U_n,n}^{(k_n)}-\Dst_{U_n,n}\prob 0, \qquad n\to \infty.
\end{align}
Using the Skorohod representation theorem, we can couple the random variables and assume the existence of a probability space on which the convergence in \eqref{thmA11} holds almost surely. For simplicity, we will use the same notation for the new random variables as for the original ones.

Let $f\colon\,(\R ,|\cdot|) \to (\R,|\cdot|)$ be an arbitrary bounded and uniformly continuous function. Then, by the uniform continuity of $f$, 
\begin{align*}
f(\Delta_{U_n,n}^{(k_n)})-f(\Dst_{U_n,n} )\as 0, \qquad n\to \infty.
\end{align*}
In particular, by the dominated convergence theorem,
\begin{align*}
\int_\R f\dd\mu_{n}^{(k_n)}-\int_\R f\dd\mu^{(\infty)}_{n} 
= \E_{\bar{\Lambda}_{n}}\Big[f(\Delta_{U_n,n}^{(k_n)})-f(\Dst_{U_n,n} ) \mid G_{n}\Big] \prob 0, \qquad n\to \infty.
\end{align*}
On the other hand, from Theorem \ref{thm3}(b),
\begin{align*}
\int_\R f\dd\mu^{(\infty)}_{n} \prob \int_\R f\dd\mu, \qquad n\to \infty.
\end{align*}
Therefore, by applying the continuous mapping theorem, we have
\begin{align*}
\int_\R f\dd\mu_{n}^{(k_n)}  \prob \int_\R f\dd\mu, \qquad n\to \infty,
\end{align*}
which, by the Portmanteau theorem, settles the claim.
\end{proof}

\begin{proof}[Proof of Theorem \ref{thmA2}]
Let $f\colon\,(\R,\vert\cdot\vert) \to (\R,\vert\cdot\vert)$ be an arbitrary bounded and uniformly continuous function. Consider $k_{n}$ such that $k_{n}\leq \frac{\log m_{n}}{2\log M}$ with $k_{n}\to\infty$ as $n\to\infty$. To complete the proof, by the Portmanteau theorem it is sufficient to show that
\begin{align}
\label{CM-new-1}
\E_{\bar{\Lambda}_{n}}\big[f(\Delta_{U_n,n}^{(k_n)}) \mid G_{n}\big]=\int_\R f\dd\mu_{n}^{(k_n)}  \weak \int_\R f\dd\mu=\E_{\bar{\Lambda}}\big[f(\Delta_{\phi})\big],\qquad n\to\infty .
\end{align}

Let $U_{n}$ be a uniformly chosen vertex in $[n]$. Also, let $\E_{\bar{\Lambda}_{n}}$ be defined as the proof of Theorem \ref{thm1}. We define:
\begin{itemize}
\item 
$(\mathrm{G}_{n}(t))_{t\in\N}$ is the graph exploration process from $U_n$, i.e., $\mathrm{G}_{n}(t)$ is the exploration where precisely $t-1$ half-edges have been paired in the breadth-first manner, that also includes the half-edges incident to the discovered vertices. In particular, $\mathrm{G}_{n}(1)$ contains $U_n$ and its $D_n=\dn_{U_n}$ half-edges. Note that every further exploration corresponds to the pairing of a half-edge, and that from $(\mathrm{G}_{n}(t))_{t\in\N}$ we can retrieve all neighbourhoods around $U_n$ in the graph $G_n$. 
\item 
$(\mathrm{BP}_{n}(t))_{t\in\N}$ is the branching process where its root, denoted by $\varnothing$, has a number of offspring distributed as $D_n$, while all other vertices have offspring distributed as $D_{n}^{\star}-1$, where $D_{n}^{\star}$ is the size-biased version of $D_n$, i.e., $\mathrm{BP}_{n}(t)$ is the branching process where  precisely $t$ tree vertices have been explored in the breadth-first order, so that $\mathrm{BP}_{n}(1)$ contains $\varnothing$ and its neighbours. Note that a tree vertex is considered explored once it has been inspected how many children it has. Also we have, as $n\to\infty$,
\begin{align}
\label{CM-new-0}
D_{n}\weak D  \qquad D_{n}^{\star}\weak D^{\star},
\end{align}
where $D^{\star}$ is defined as in \eqref{CM-neq--2}.
\end{itemize}

According to \cite[Lemma 4.2]{RvdH2}, there exists a coupling $(\hat{\mathrm{G}}_{n}(t),\hat{\mathrm{BP}}_{n}(t))_{t\in\N}$ of $(\mathrm{G}_{n}(t))_{t\in\N}$ and $(\mathrm{BP}_{n}(t))_{t\in\N}$ such that
\begin{align}
\label{CM-new-2}
\P_{n}\Big\{(\hat{\mathrm{G}}_{n}(t))_{t\in[m_n]}\neq (\hat{\mathrm{BP}}_{n}(t))_{t\in[m_n]}\Big\}=o(1),\qquad n\to\infty .
\end{align}
In the following, we consider $n$ large enough and run finite non-backtracking random walks on $(\hat{\mathrm{G}}_{n}(t))_{t\in\N}$ and $(\hat{\mathrm{BP}}_{n}(t))_{t\in\N}$ (and later also on $(\mathrm{BP}_{n}(t))_{t\in\N}$). To avoid redefining the same object for different graphs or processes, we will use a suffix to indicate the corresponding version. For any $\epsilon>0$, we have
\begin{align*}
\P_{n}\Big\{\big|\Delta_{U_n,n,\hat{\mathrm{G}}_{n}}^{(k_n)}-\Delta_{U_n,n,\hat{\mathrm{BP}}_{n}}^{(k_n)}\big|>\epsilon\Big\} \leq \P_{n}\Big\{(\hat{\mathrm{G}}_{n}(t))_{t\in[k_n]}\neq (\hat{\mathrm{BP}}_{n}(t))_{t\in[k_n]}\Big\}.
\end{align*}
because the number of tree vertices that can be explored up to $k_n$th-generation of the branching process $(\hat{\mathrm{BP}}_{n}(t))_{t\in\N}$ is at most $k_{n}M^{k_n}\leq M^{2k_n}$ for $n$ large enough, which is less than or equal to $m_n$. Hence, by \eqref{CM-new-2},
\begin{align*}
\Delta_{U_n,n,\hat{\mathrm{G}}_{n}}^{(k_n)}-\Delta_{U_n,n,\hat{\mathrm{BP}}_{n}}^{(k_n)} \prob 0,\qquad n\to\infty.
\end{align*}
But $\Delta_{U_n,n,\hat{\mathrm{G}}_{n}}^{(k_n)}$ is simply a copy of $\Delta_{U_n,n}^{(k_n)}$, and $\Delta_{U_n,n,\hat{\mathrm{BP}}_{n}}^{(k_n)}$ is likewise a copy of $\Delta_{\varnothing,n,\mathrm{BP}_{n}}^{(k_n)}$ where $\varnothing$ is the uniformly chosen vertex $U_n$. Therefore, keeping $\varnothing=U_n$, we have
\begin{align}
\label{CM-neww-2}
\Delta_{U_n,n}^{(k_n)}-\Delta_{\varnothing,n,\mathrm{BP}_{n}}^{(k_n)} \prob 0,\qquad n\to\infty.
\end{align}

From \eqref{CM-new-0}, we have
\begin{align}
\label{CM-neww-1}
d_{\varnothing}^{(\mathrm{BP}_{n})}\weak d_{\phi}.
\end{align}
 Since $d_{j}^{(\mathrm{BP}_{n})}$ and $P_{n,\mathrm{BP}_{n}}^{(k_n)}(\varnothing, j)$ are independent for each $j\in V_{k_n}(\mathrm{BP}_{n})$, where $V_{k_n}(\mathrm{BP}_{n})$ is the set of all vertices in the branching process $(\mathrm{BP}_{n}(t))_{t\in\N}$ up to and including generation $k_n$, we have, by the dominated convergence theorem and \eqref{CM-new-0},
\begin{align*}
\E_n\bigg[\sum_{j\in V_{k_n}(\mathrm{BP}_{n})} P_{n,\mathrm{BP}_{n}}^{(k_n)}(\varnothing, j)\, 
d^{(\mathrm{BP}_{n})}_{j}\bigg] = \E[D_{n}^{\star}]\to \E[D^{\star}],\qquad n\to\infty .
\end{align*}
Moreover, since $P_{n,\mathrm{BP}_{n}}^{(k_n)}(\varnothing, j)\,P_{n,\mathrm{BP}_{n}}^{(k_n)}(\varnothing, l)$ is non-zero only if both $j$ and $l$ are in the $k_n$th-generation and $j \neq l$, we must have that $P_{n,\mathrm{BP}_{n}}^{(k_n)}(\varnothing, j)P_{n,\mathrm{BP}_{n}}^{(k_n)}(\varnothing, l)$ is independent of $d^{(\mathrm{BP}_{n})}_{j}$ and $d^{(\mathrm{BP}_{n})}_{l}$. Since $\sup_{j\in V_{k_n}(\mathrm{BP}_{n})}P_{n,\mathrm{BP}_{n}}^{(k_n)}(\varnothing, j)\leq 2^{-k_n}$ we have, by the dominated convergence theorem and \eqref{CM-new-0},
\begin{align*}
\E_n\bigg[\Big(\sum_{j\in V_{k_n}(\mathrm{BP}_{n})} P_{n,\mathrm{BP}_{n}}^{(k_n)}(\varnothing, j)\, 
d^{(\mathrm{BP}_{n})}_{j}\Big)^2\bigg]
\leq M^2 \, 2^{-k_n}+\big(\E[D_{n}^{\star}]\big)^2 \to \big(\E[D^{\star}]\big)^2,\qquad n\to \infty.
\end{align*}
Therefore, by the Chebyshev inequality,
\begin{align*}
\sum_{j\in V_{k_n}(\mathrm{BP}_{n})} P_{n,\mathrm{BP}_{n}}^{(k_n)}(\varnothing, j)\, 
d^{(\mathrm{BP}_{n})}_{j}\prob \E[D^{\star}]=\dfrac{\E[D^2]}{\E[D]},\qquad n\to\infty.
\end{align*}
Hence, by the continuous mapping theorem and \eqref{CM-neww-1},
\begin{align*}
\Delta_{\varnothing ,n,\mathrm{BP}_{n}}^{(k_n)}\weak \Delta_{\phi},\qquad n\to\infty,
\end{align*}
and, by the continuity of $f$ and the dominated convergence theorem,
\begin{align}
\label{CM-newww}
\E\big[f(\Delta_{\varnothing ,n,\mathrm{BP}_{n}}^{(k_n)})\big]
\,\to\, 
\E_{\bar{\Lambda}}\big[f(\Delta_{\phi})\big],\qquad n\to\infty.
\end{align}
Applying the Skorohod representation theorem, we can couple the random variables such that the convergence in \eqref{CM-neww-2} is almost sure. By the uniform continuity of $f$ we can therefore conclude that 
\begin{align*}
f(\Delta_{U_n,n}^{(k_n)}) - f(\Delta_{\varnothing ,n,\mathrm{BP}_{n}}^{(k_n)}) \prob 0, \qquad n\to\infty,
\end{align*}
and so, by the dominated convergence theorem, 
\begin{align*}
\E_{\bar{\Lambda}_{n}}\big[f(\Delta_{U_n,n}^{(k_n)}) \mid G_{n}\big] 
- \E\big[f(\Delta_{\varnothing ,n,\mathrm{BP}_{n}}^{(k_n)})\big] \prob 0,\qquad n\to\infty.
\end{align*}
This, together with \eqref{CM-newww} and the continuous mapping theorem, implies \eqref{CM-new-1} and establishes the theorem.
\end{proof}

\begin{proof}[Proof of Theorem \ref{thm4}]
Consider a sequence $(k_n)_{n\in\N}$ satisfying $\lim_{n\to\infty} k_n = \infty$. From the assumptions of the theorem and the Skorohod representation theorem, we can couple the random variables on a probability space $(\Omega,\mathcal{F},\mathbb{P})$ such that
\begin{align*}
\Psi(N)\as 0, \qquad N\to\infty,
\end{align*}
and
\begin{align}
\label{thm4-1}
\mu_{n}^{(\infty)} \weak \mu, \qquad n\to\infty \qquad \text{almost surely.}
\end{align}
Here, for simplicity, we have used the same notation for the new random variables as for the original ones. Since weak convergence in Polish spaces is equivalent to convergence in the Prohorov metric \cite[Theorem 6.8]{Billingsley}, from \eqref{thm4-1} we have
\begin{align*}
\mathrm{d}_{P}(\mu_{n}^{(\infty)},\mu )\as 0,\qquad n\to\infty. 
\end{align*}
As a result, there exists $\Omega^{\prime}\subseteq \Omega$ with $\P\{\Omega^{\prime}\}=1$ such that for any $\omega\in\Omega^{\prime}$ and any $\epsilon>0$, we can find $N_{\epsilon}\in\N$ such that, for all $n,N\geq N_{\epsilon}$,
\begin{align*}
\mathrm{d}_{P}(\mu_{n}^{(k_n)},\mu )\leq \mathrm{d}_{P}(\mu_{n}^{(k_n)},\mu_{n}^{(\infty)} )
+ \mathrm{d}_{P}(\mu_{n}^{(\infty)},\mu )\leq \Psi(N)+\mathrm{d}_{P}(\mu_{n}^{(\infty)},\mu ) < \epsilon.
\end{align*}
Therefore $\mathrm{d}_{P}(\mu_{n}^{(k_n)},\mu)\as 0$ as $n\to\infty$ on $(\Omega , \mathcal{F},\mathbb{P})$, which implies that in the original probability space, $\mu_{n}^{(k_n)} \weak \mu$ as $n\to\infty$ in probability, which completes the proof of the theorem.
\end{proof}

%%%%%%% REFERENCES %%%%%%%%%%%%%%%%%%%%%

%%%%%%%%%%%%%%%%%%%%%%%%%%%%%%%%%%%%%%%%%%%%%


\begin{thebibliography}{99}
\bibitem{AL}
D.\ Aldous, R.\ Lyons,
Processes on unimodular random networks,
Electronic Journal of Probability 12 (2007), 1454--1508.

\bibitem{BenHS}
A.\ Ben-Hamou, J.\ Salez,
Cutoff for nonbacktracking random walks on sparse random graphs,
The Annals of Probability 45 (2017), 1752--1770.

\bibitem{BS}
I.\ Benjamini, O.\ Schramm,
Recurrence of distributional limits of finite planar graphs,
Electronic Journal of Probability 6 (2001), 1--13.

\bibitem{BZ}
K.S.\ Berenhaut, C.M.\ Zhang, 
Disparity-persistence and the multistep friendship paradox,
Probability in the Engineering and Informational Sciences (2023), 1--9.

\bibitem{Billingsley}
P.\ Billingsley, 
\textit{Convergence of Probability Measures}, Second Edition, John Wiley \& Sons, New York, 1999.

\bibitem{B}
P.\ Br\'{e}maud,
\textit{Markov Chains: Gibbs Fields, Monte Carlo Simulation, and Queues}, Volume 31,
Springer Science \& Business Media, 2013.

\bibitem{CKN}
G.T.\ Cantwell, A. Kirkley, and M.E.J.\ Newman,
The friendship paradox in real and model networks,
Journal of Complex Networks 9 (2021), cnab011. 

\bibitem{CR}
Y.\ Cao, S.M.\ Ross,
The friendship paradox, 
Mathematical Scientist 41 (2016), 61--64. 

\bibitem{CF}
N.A.\ Cristakis, J.H.\ Fowler, 
The collective dynamics of smoking in a large social network,
New England Journal of Medicine 358 (2008), 2249--2258.

\bibitem{EH}
I.\ Eisner, S.\ Hoory,
Entropy and the growth rate of universal covering trees,
\url{https://arxiv.org/pdf/2410.10337}.

\bibitem{EJ} 
Y.H.\ Eom, H.H.\ Jo, 
Generalized friendship paradox in complex networks: the case of scientific collaboration, 
Scientific Reports 4 (2014), 1--6.

\bibitem{SF}
S.\ Feld,
Why your friends have more friends than you do, 
American Journal of Sociology 96 (1991), 1464--1477.

\bibitem{HHP}
R.S.\ Hazra, F.\ den Hollander, A.\ Parvaneh,
The friendship paradox for sparse random graphs, 
Probability Theory and Related Fields (2025), \url{https://doi.org/10.1007/s00440-025-01365-w}.

\bibitem{HKL} 
N.O.\ Hodas, F.\ Kooti, K.\ Lerman,
Friendship paradox redux: your friends are more interesting than you,
Proceedings of the International AAAI Conference on Web and Social Media 7 (2013), 225--233.

\bibitem{RvdH1}
R.\ van der Hofstad,
\textit{Random Graphs and Complex Networks}, Volume 1, 
Cambridge Series in Statistical and Probabilistic Mathematics, 2017.

\bibitem{RvdH2}
R.\ van der Hofstad,
\textit{Random Graphs and Complex Networks}, Volume 2,
Cambridge Series in Statistical and Probabilistic Mathematics, 2024. 

\bibitem{J} 
M.O.\ Jackson, 
The friendship paradox and systematic biases in perceptions and social norms, 
Journal of Political Economy 127 (2019), 777--818.

\bibitem{K}
M.\ Kempton,
Non-backtracking random walks and a weighted Ihara's theorem, 
Open Journal of Discrete Mathematics 6 (2016), 207--226.

\bibitem{Kh}
A.\ Khezeli,
Unimodular random measured metric spaces and Palm theory on them,
\url{https://arxiv.org/abs/2304.02863}.

\bibitem{KCR}
J.B.\ Kramer, J.\ Cutler, A.J.\ Radcliffe,  
The multistep friendship paradox, 
The American Mathematical Monthly 123 (2016), 900--908.

\bibitem{NK} 
B.\ Nettasinghe and V.\ Krishnamurthy, 
``What do your friends think?": efficient polling methods for networks using friendship paradox,
IEEE Transactions on Knowledge and Data Engineering 33 (2019), 1291--1305.

\bibitem{N}
J.R.\ Norris,
\textit{Markov Chains}, No. 2,
Cambridge University Press, 1998.

\bibitem{PYNSB}
S.\ Pal, F.\ Yu, Y.\ Novick, A.\ Swami and A.\ Bar-Noy,
A study on the friendship paradox -- quantitative analysis and relationship with assortative mixing,
Applied Network Science 4 (2019), paper 71.

\end{thebibliography}
\end{document}